\documentclass[a4paper,10pt]{amsart}
\usepackage[utf8x]{inputenc}
\usepackage{amsmath}
\usepackage{amsthm}
\usepackage{amssymb}
\usepackage{hyperref}
\usepackage{tikz}
\usepackage{amsfonts}
\usepackage[framemethod=TikZ]{mdframed}
\usetikzlibrary{positioning,arrows}
\usetikzlibrary{matrix}
\usetikzlibrary{shapes,snakes}
\usetikzlibrary{arrows,decorations.pathreplacing,patterns}

\mdfdefinestyle{Example}{%
    linecolor=blue,
    outerlinewidth=1pt,
    roundcorner=2pt,
    innertopmargin=\baselineskip,
    innerbottommargin=\baselineskip,
    innerrightmargin=20pt,
    innerleftmargin=20pt,
    backgroundcolor=gray!25!white}
    
\newtheorem{theorem}{Theorem}[section]
\newtheorem{lemma}[theorem]{Lemma}

\newtheorem{corollary}[theorem]{Corollary}
\newtheorem*{theorem*}{Theorem}
\newtheorem*{corollary*}{Corollary}

\newcommand{\defn}[1]{\emph{#1}}

\def\Z{\mathbb{Z}}
\def\R{\mathbb{R}}
\def\N{\mathbb{N}}
\def\simp{\Delta}
\def\wa{\widetilde{W}}
\def\ac{A_{\circ}}
\def\Q{Q^{\vee}}
\def\m{m}
\def\I{\mathcal{I}}
\def\J{\mathcal{J}}
\def\H{\mathcal{H}}
\def\alc{\mathsf{Alc}}
\def\park{\mathsf{Park}_{\Phi}}
\def\parkm{\mathsf{Park}_{\Phi}^{(m)}}
\def\pta{\Theta}
\def\ptt{\Gamma}
\def\A{\mathcal{A}}

\def\rd{R_{\rm{dom}}}
\def\re{R'}
\def\rmin{R_{\rm{min}}}
\def\ar{\widetilde{\Phi}}

\def\wm{\waf_{\rm{min}}}
\newcommand{\op}[1]{\operatorname{#1}}
\def\pf{\mathcal{PF}}
\def\D{\mathcal{D}}
\newcommand{\rat}{p}
\def\h{\mathrm{ht}}

\def\lvl{\mathsf{levels}}
\def\modn{\text{ }\mathrm{mod}\text{ }n}
\newcommand{\packrat}{\pf_{\rat/n}}
\newcommand{\sommers}{\mathcal{S}_{\Phi}^{\rat}}
\newcommand{\wf}{\waf_{\rat}}

\newcommand{\waf}{\widetilde{w}}
\newcommand{\cwl}{\Lambda^{\vee}}
\newcommand{\cw}{\omega^{\vee}}

\newcommand{\inv}{\mathsf{Inv}}
\newcommand{\wratdom}{\wa^{\rat}_{\mathrm{dom}}}

\newcommand{\drw}{\mathsf{drw}}
\newcommand{\ind}{\op{ind}}

\newcommand{\ptd}{\epsilon}
\newcommand{\atd}{\delta}
\newcommand{\ttpf}{\chi}

\newcommand{\summed}{\mathsf{Sum}}

\newcommand{\hr}{\theta}
\begin{document}

\title{From Anderson to Zeta}
\author{Marko Thiel}
\address{Department of Mathematics, University of Zürich, Winterthurerstrasse 190, 8057 Zürich, Switzerland}

\begin{abstract} For an irreducible crystallographic root system $\Phi$ and a positive integer $p$ relatively prime to the Coxeter number $h$ of $\Phi$,
we give a natural bijection $\mathcal{A}$ from the set $\widetilde{W}^{p}$ of affine Weyl group elements with no inversions of height $p$ to the finite torus $\check{Q}/p\check{Q}$. Here $\check{Q}$ is the coroot lattice of $\Phi$.
This bijection is defined uniformly for all irreducible crystallographic root systems $\Phi$ and is equivalent to the \emph{Anderson map} $\mathcal{A}_{GMV}$ defined by Gorsky, Mazin and Vazirani when $\Phi$ is of type $A_{n-1}$.\\
\\
Specialising to $p=mh+1$, we use $\mathcal{A}$ to define a uniform $W$-set isomorphism $\zeta$ from the finite torus $\check{Q}/(mh+1)\check{Q}$ to the set of $m$-nonnesting parking functions $\mathsf{Park}_{\Phi}^{(m)}$ of $\Phi$.
The map $\zeta$ is equivalent to the \emph{zeta map} $\zeta_{HL}$ of Haglund and Loehr when $m=1$ and $\Phi$ is of type $A_{n-1}$.
%
%
\end{abstract}
%
\maketitle
\section{Introduction}
The aim of this article is to describe uniform generalisations to all irreducible crystallographic root systems $\Phi$ of two bijections that arose from the study of the Hilbert series $\mathcal{DH}(n;q,t)$ of the space of diagonal harmonics of the symmetric group $S_n$.
These bijections are the \defn{Anderson map} $\A_{GMV}$ of Gorsky, Mazin and Vazirani \cite[Section 3.1]{gorsky14affine} and the \defn{zeta map} $\zeta_{HL}$ of Haglund and Loehr \cite[Theorem 5.6]{haglund08catalan}.
\subsection{The Hilbert series of the space of diagonal harmonics}
The space of diagonal harmonics of the symmetric group $S_n$ is a well-studied object in algebraic combinatorics.
Its Hilbert series has two combinatorial interpretations \cite[Conjecture 5.2]{haglund08catalan}, \cite{carlsson16shuffle}: 
$$\mathcal{DH}(n;q,t)=\sum_{(P,\sigma)\in\pf_n}q^{\op{dinv}^\prime(P,\sigma)}t^{\op{area}(P,\sigma)}=\sum_{(w,D)\in\D_n}q^{\op{area}^\prime(w,D)}t^{\op{bounce}(w,D)},$$
where $\pf_n$ is the set of \defn{parking functions} of length $n$ viewed as vertically labelled Dyck paths and $\D_n$ is the set of \defn{diagonally labelled Dyck paths} of length $n$.
There is a bijection $\zeta_{HL}$ due to Haglund and Loehr \cite[Theorem 5.6]{haglund08catalan} that maps $\pf_n$ to $\D_n$ and sends the bistatistic $(\op{dinv}^\prime,\op{area})$ to $(\op{area}^\prime,\op{bounce})$, demonstrating the second equality.\\
\\
In their study of \defn{rational parking functions}, Gorsky, Mazin and Vazirani introduced the \defn{Anderson map} $\A_{GMV}$ as a bijection from the set of \defn{$\rat$-stable affine permutations} $\widetilde{S}^{\rat}_n$ to the set of $\rat/n$-parking functions $\pf_{\rat/n}$. Here $\rat$ is any positive integer relatively prime to $n$ and $\pf_{n+1/n}=\pf_n$.
They used it to define a combinatorial Hilbert series for the set of $\rat$-stable affine permutations that generalises $\mathcal{DH}(n;q,t)$ and related $\A_{GMV}$ to $\zeta_{HL}$ \cite[Definition 3.26, Theorem 5.3]{gorsky14affine}.
Their approach builds on the work by Gorsky and Mazin on compactified Jacobians of plane curve singularities with one Puiseux pair \cite{gorsky13jacobian,gorsky14jacobian}.
\subsection{Beyond type $A_{n-1}$}
One can view all the objects $\pf_n$, $\D_n$, $\widetilde{S}_n^{\rat}$ and $\pf_{\rat/n}$ as well as the maps $\zeta_{HL}$ and $\A_{GMV}$ as being associated with the root system of type $A_{n-1}$.
We will generalise both the zeta map $\zeta_{HL}$ of Haglund and Loehr and the Anderson map $\A_{GMV}$ of Gorsky, Mazin and Vazirani to all irreducible crystallographic root systems $\Phi$.
We work at three different levels of generality, in order from most general to least general these are the \defn{rational} level, the \defn{Fuß-Catalan} level, and the \defn{Coxeter-Catalan} level.
All the objects we now mention will be defined in later sections.
\subsection{The rational level}
We fix a positive integer $\rat$ that is relatively prime to the Coxeter number $h$ of $\Phi$.
We define the \defn{Anderson map} $\A$ as a bijection from the set $\wa^{\rat}$ of \defn{$\rat$-stable affine Weyl group elements} to the \defn{finite torus} $\Q/\rat\Q$. Here $\Q$ is the coroot lattice of $\Phi$.
If $\Phi$ is of type $A_{n-1}$ it reduces to the Anderson map $\A_{GMV}$ that maps the set $\widetilde{S}_n^{\rat}$ of $\rat$-stable affine permutations to the set of rational parking functions $\pf_{\rat/n}$.
\subsection{The Fuß-Catalan level}
At the Fuß-Catalan level, we specialise to $\rat=mh+1$ for some positive integer $m$.
We consider an affine hyperplane arrangement in the ambient space of $\Phi$ called the \defn{$m$-Shi arrangement}.
Every region $R$ of that arrangement has a unique minimal alcove $\waf_R\ac$, and the set of minimal alcoves of regions of the $m$-Shi arrangement corresponds to $\wa^{mh+1}$.
The Anderson map $\A$ gives a bijection from $\wa^{mh+1}$ to the finite torus $\Q/(mh+1)\Q$.\\
\\
The set $\park^{(\m)}$ of \defn{$m$-nonnesting parking functions} was defined by Rhoades \cite{rhoades12parking} as a model for the set of regions of the $m$-Shi arrangement that carries an action of the \defn{Weyl group} $W$.
The finite torus $\Q/(mh+1)\Q$ also has a natural $W$-action.
Using the inverse $\A^{-1}$ of the Anderson map we define a $W$-set isomorphism $\zeta$ from $\Q/(mh+1)\Q$ to $\park^{(\m)}$. We call this the zeta map $\zeta$.
\subsection{The Coxeter-Catalan level}
At the Coxeter-Catalan level, we specialise further to $m=1$. Thus we have $\rat=h+1$.
In the case where $\Phi$ is of type $A_{n-1}$, we identify the combinatorial objects $\pf_n$ and $\D_n$ 
with the finite torus $\Q/(h+1)\Q$ and the set $\park$ of nonnesting parking functions of $\Phi$ respectively.
With these identifications, our zeta map $\zeta$ coincides with the zeta map $\zeta_{HL}$.
\subsection{Structure of the article}
In Section \ref{bassec} we give background on finite and affine crystallographic root systems and their Weyl groups.
Sections \ref{classical} to \ref{andgmv} are devoted to defining the combinatorial Anderson map $\A_{GMV}$.
Section \ref{and} defines the uniform Anderson map $\A$ as a composition of various maps that have already appeared in the literature in some form.
Then Section \ref{and+and} shows that in the case where $\Phi$ is of type $A_{n-1}$, the uniform Anderson map $\A$ is equivalent to the combinatorial Anderson map $\A_{GMV}$.\\
\\
Sections \ref{shiarr} and \ref{minsec} introduce the $\m$-Shi arrangement and show that its minimal alcoves correspond to the $(mh+1)$-stable affine Weyl group elements in $\wa^{mh+1}$.
Then Sections \ref{diag} and \ref{zetahl} define the combinatorial zeta map $\zeta_{HL}$ due to Haglund and Loehr.
Section \ref{zetasec} defines the uniform zeta map $\zeta$ and then Section \ref{zeta+zeta} shows that it reduces to $\zeta_{HL}$ if $m=1$ and $\Phi$ is of type $A_{n-1}$.
\section{Root systems and their Weyl groups}\label{bassec}
Let $\Phi$ be an irreducible (finite) crystallographic root system of rank $r$ with ambient space $V$. So $V$ is a real vector space of dimension $r$ with an inner product $\langle\cdot,\cdot\rangle$. For background on root systems see \cite{humphreys90reflection}. Choose a set of simple roots $\simp$ for $\Phi$ and let $\Phi^+$ be the corresponding set of positive roots.
Let $W$ be the Weyl group of $\Phi$ and let $S$ be the set of simple reflections corresponding to $\simp$. Then $S$ generates $W$ and $(W,S)$ is a Coxeter system. For a positive integer $n$, write $[n]:=\{1,2,\ldots,n\}$.\\
\begin{mdframed}[style=Example]
\textbf{Example.} The root system of type $A_{n-1}$ is 
\[\Phi=\{e_i-e_j:i,j\in[n],i\neq j\}.\]
It has rank $n-1$ and ambient space 
\[V=\{(x_1,x_2,\ldots,x_n)\in\R^n:\sum_{i=1}^nx_i=0\}.\]
We let $\alpha_i=e_i-e_{i+1}$ for $i\in[n-1]$ and choose $\Delta=\{\alpha_1,\alpha_2,\ldots,\alpha_{n-1}\}$.
The Weyl group $W$ of $\Phi$ is the symmetric group $S_n$ on $n$ letters that acts on $V$ by permuting coordinates.
It is generated by $S=\{s_1,s_2,\ldots,s_{n-1}\}$, where $s_i=(i\text{ }i+1)$ is the $i$-th adjacent transposition, corresponding to the reflection through the linear hyperplane orthogonal to $\alpha_i$.
\end{mdframed}
\vspace{1em}
For $k\in\Z$ and $\alpha\in\Phi$, define the affine hyperplane
$$H_{\alpha}^k:=\{x\in V: \langle x,\alpha\rangle=k\}.$$
\subsection{The Coxeter arrangement}
The \defn{Coxeter arrangement} is the central hyperplane arrangement in $V$ given by all the linear hyperplanes $H_{\alpha}:=H_{\alpha}^0$ for $\alpha\in\Phi^+$.
The complement of this arrangement falls apart into connected components which we call \defn{chambers}.
The Weyl group $W$ acts simply transitively on the chambers. Thus we define the \defn{dominant chamber} by 
$$C:=\{x\in V: \langle x,\alpha\rangle>0\text{ for all }\alpha\in\simp\}$$
and write any chamber as $wC$ for a unique $w\in W$. 
\subsection{The affine Coxeter arrangement and the affine Weyl group}
From now on we will assume that $\Phi$ is irreducible.
The \defn{root order} on $\Phi^+$ is the partial order defined by 
$\alpha\leq\beta$ if and only if $\beta-\alpha$ can be written as a sum of positive roots.
The set of positive roots $\Phi^+$ with this partial order is called the \defn{root poset} of $\Phi$.
It has a unique maximal element, the \defn{highest root} $\hr$.\\
\\
The \defn{affine Coxeter arrangement} is the affine hyperplane arrangement in $V$ given by all the affine hyperplanes $H_{\alpha}^k$ for $\alpha\in\Phi$ and $k\in\Z$.
The complement of this arrangement falls apart into connected components which are called \defn{alcoves}. They are all isometric simplices. We call an alcove \defn{dominant} if it is contained in the dominant chamber.
Define $s_{\alpha}^k$ as the reflection through the affine hyperplane $H_{\alpha}^k$. That is,
$$s_{\alpha}^k(x):=x-2\frac{\langle x,\alpha\rangle-k}{\langle\alpha,\alpha\rangle}\alpha.$$
We will also write $s_{\alpha}$ for the linear reflection $s_{\alpha}^0$.\\
\\
Let the \defn{affine Weyl group} $\wa$ be the group of affine automorphisms of $V$ generated by all the reflections through hyperplanes in the affine Coxeter arrangement,
that is
$$\wa:=\langle s_{\alpha}^k: \alpha\in\Phi\text{ and }k\in\Z\rangle.$$
The affine Weyl group $\wa$ acts simply transitively on the alcoves of the affine Coxeter arrangement. Thus we define the \defn{fundamental alcove} by
$$\ac:=\{x\in V: \langle x,\alpha\rangle>0\text{ for all }\alpha\in\simp\text{ and }\langle x,\hr\rangle<1\}$$
and write any alcove of the affine Coxeter arrangement as $\waf\ac$ for a unique $\waf\in\wa$.\\
\\
If we define $\widetilde{S}:=S\cup\{s_{\hr}^1\}$, then $(\wa,\widetilde{S})$ is a Coxeter system. \\
In particular, we may write any $\waf\in\wa$ as a word in the generators in $\widetilde{S}$, called reduced if its length is minimal, and define the length $l(\waf)$ of $\waf$ as the length of any reduced word for it.\\
\\
For a root $\alpha\in\Phi$, its \defn{coroot} is defined as $\alpha^{\vee}=2\frac{\alpha}{\langle\alpha,\alpha\rangle}$.
The set $\Phi^{\vee}=\{\alpha^{\vee}:\alpha\in\Phi\}$ is itself an irreducible crystallographic root system, called the \defn{dual root system} of $\Phi$.
Clearly $\Phi^{\vee\vee}=\Phi$.\\
\\
The \defn{root lattice} $Q$ of $\Phi$ is the lattice in $V$ spanned by all the roots in $\Phi$.
The \defn{coroot lattice} $\Q$ of $\Phi$ is the lattice in $V$ spanned by all the coroots in $\Phi^{\vee}$.
It is not hard to see that $\wa$ acts on the coroot lattice. To any $\mu\in\Q$, there corresponds the translation 
\begin{align*}
 t_{\mu}:V&\rightarrow V\\
 x&\mapsto x+\mu.
\end{align*}
\begin{figure}[h]
\begin{center}
 \resizebox*{9cm}{!}{\includegraphics{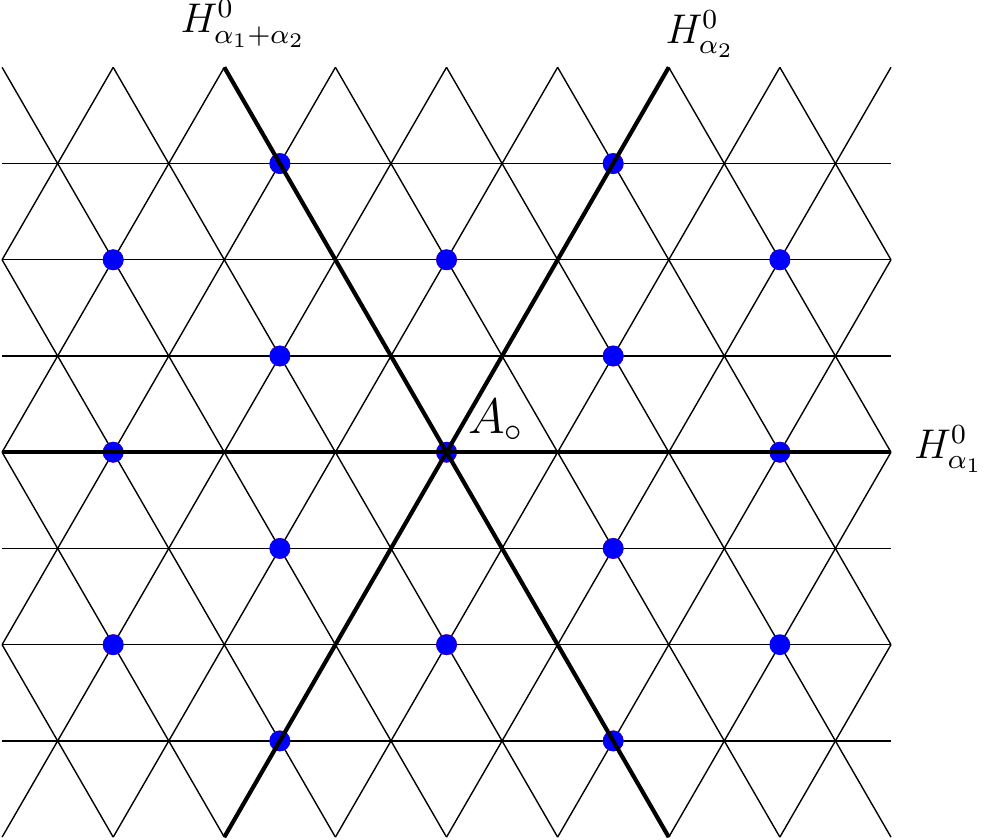}}\\
\end{center}
\caption{The affine Coxeter arrangement of type $A_2$. The points in the coroot lattice $\Q$ are marked as blue dots.}
\end{figure}
\\
If we identify $\Q$ with the corresponding group of translations acting on the affine space $V$ then we may write $\wa=W\ltimes\Q$ as a semidirect product.
In particular, we may write any $\waf\in\wa$ as $\waf=wt_{\mu}$ for unique $w\in W$ and $\mu\in\Q$.\\
\\
For an alcove $\waf\ac$ and a root $\alpha\in\Phi$ there is a unique integer $k$ such that $k<\langle x,\alpha\rangle<k+1$ for all $x\in \waf\ac$.
We denote this integer by $k(\waf,\alpha)$. We call the tuple $(k(\waf,\alpha))_{\alpha\in\Phi^+}$ the \defn{address} of the alcove $\waf\ac$.\\
\\
Notice that we have $k(\waf,-\alpha)=-k(\waf,\alpha)-1$ and $k(w\waf,w(\alpha))=k(\waf,\alpha)$ for all $\alpha\in\Phi$ and $w\in W$. Also note that if $k(\waf,\alpha)=k(\waf',\alpha)$ for all $\alpha\in\Phi^+$, then $\waf=\waf'$.
\subsection{Polyhedra}
A \defn{half-space} in $V$ is subset of the form
\[\H=\{x\in V:\lambda(x)<l\}\]
for some linear functional $\lambda$ in the dual space $V^{*}$ of $V$ and a real number $l\in\R$.
A \defn{polyhedron} in $V$ is any nonempty subset of $V$ that is defined as the intersection of a finite number of half-spaces.
We will also consider the closure $\overline{P}$ of any polyhedron $P$ a polyhedron.\\
\\
Any hyperplane corresponding to an irredundant inequality defining a polyhedron $P$ is considered a \defn{wall} of $P$.
Its intersection with $\overline{P}$ is called a \defn{facet} of $P$.
A wall of $P$ is called a \defn{floor} of $P$ if it does not contain the origin and separates $P$ from the origin.
\subsection{Affine roots}
We may understand $\wa$ in terms of its action on the set of \defn{affine roots} $\ar$ of $\Phi$. To do this, let $\delta$ be a formal variable and define $\widetilde{V}:=V\oplus\R\delta$.
Define the set of affine roots as
$$\ar:=\{\alpha+k\delta:\alpha\in\Phi\text{ and }k\in\Z\}.$$
If $\waf\in\wa$, write it as $\waf=wt_{\mu}$ for unique $w\in W$ and $\mu\in\Q$ and define
$$\waf(\alpha+k\delta)=w(\alpha)+(k-\langle\mu,\alpha\rangle)\delta.$$
This defines an action of $\wa$ on $\ar$. It imitates the action of $\wa$ on the half-spaces of $V$ defined by the hyperplanes of the affine Coxeter arrangement.
To see this, define the half-space
$$\H_{\alpha+k\delta}:=\{x\in V:\langle x,\alpha\rangle>-k\}.$$
Then for $\waf\in\wa$ we have $\waf(\H_{\alpha+k\delta})=\H_{\beta+l\delta}$ if and only if $\waf(\alpha+k\delta)=\beta+l\delta$.
Define the set of \defn{positive affine roots} as
$$\ar^+:=\{\alpha+k\delta:\alpha\in\Phi^+\text{ and }k\geq0\}\cup\{\alpha+k\delta:\alpha\in-\Phi^+\text{ and }k>0\},$$
the set of affine roots corresponding to half-spaces that contain $\ac$. So $\ar$ is the disjoint union of $\ar^+$ and $-\ar^+$.
Define the set of \defn{simple affine roots} as
$$\widetilde{\simp}:=\simp\cup\{-\hr+\delta\},$$
the set of affine roots corresponding to half-spaces that contain $\ac$ and share one of its walls. We will also write $\alpha_0:=-\hr+\delta$.\\
\\
For $\waf\in\wa$, we say that $\alpha+k\delta\in\ar^+$ is an \defn{inversion} of $\waf$ if $\waf(\alpha+k\delta)\in-\ar^+$.
Define 
\[\inv(\waf):=\ar^+\cap\waf^{-1}(-\ar^+)\]
as the set of inversions of $\waf$.\\
\\
The following lemmas are well-known.
\begin{lemma}\label{arsep}
 The positive affine root $\alpha+k\delta\in\ar^+$ is an inversion of $\waf$ if and only if the hyperplane $H_{\alpha}^{-k}$ separates $\waf^{-1}\ac$ from $\ac$.
 \begin{proof}
  If $\alpha+k\delta\in\ar^+\in\inv(\waf)$, then $\ac\subseteq\H_{\alpha+k\delta}$ and $\ac\nsubseteq \waf(\H_{\alpha+k\delta})$.
  Thus $\waf^{-1}\ac\nsubseteq\H_{\alpha+k\delta}$ and therefore $H_{\alpha}^{-k}$ separates $\waf^{-1}\ac$ from $\ac$.\\
  \\
  Conversely, if $\alpha+k\delta\in\ar^+$ and $H_{\alpha}^{-k}$ separates $\waf^{-1}\ac$ from $\ac$, then $\ac\subseteq\H_{\alpha+k\delta}$ and $\waf^{-1}\ac\nsubseteq\H_{\alpha+k\delta}$.
  Therefore $\ac\nsubseteq \waf(\H_{\alpha+k\delta})$ and thus $\waf(\alpha+k\delta)\in-\ar^+$. So $\alpha+k\delta\in\inv(\waf)$.
 \end{proof}

\end{lemma}
\begin{lemma}\label{arfl}
 If $\alpha+k\delta\in\ar^+$, $k>0$ and $\waf\in\wa$, then $\waf^{-1}(\alpha+k\delta)\in-\widetilde{\simp}$ if and only if $H_{\alpha}^{-k}$ is a floor of $\waf\ac$.
 \begin{proof}
  For the forward implication, suppose $\alpha+k\delta\in\ar^+$, $k>0$, $\waf\in\wa$ and $\waf^{-1}(\alpha+k\delta)\in-\widetilde{\simp}$. Then by Lemma \ref{arsep} the hyperplane $H_{\alpha}^{-k}$ separates $\waf\ac$ from $\ac$.
  But we also have that $\waf^{-1}(\H_{\alpha+k\delta})$ shares a wall with $\ac$, so $H_{\alpha}^{-k}$ is a wall of $\waf\ac$. Thus it is a floor of $\waf\ac$.\\
  \\
  Conversely, if $\alpha+k\delta\in\ar^+$ and $H_{\alpha}^{-k}$ is a floor of $\waf\ac$, then by Lemma \ref{arsep}, $\waf^{-1}(\alpha+k\delta)\in-\ar^+$, so $\ac\nsubseteq \waf^{-1}(\H_{\alpha+k\delta})$.
  But since $\H_{\alpha+k\delta}$ shares a wall with $\waf\ac$, $\waf^{-1}(\H_{\alpha+k\delta})$ shares a wall with $\ac$. So $\waf^{-1}(\alpha+k\delta)\in-\simp$.
 \end{proof}

\end{lemma}
\subsection{The height of roots}
For $\alpha\in\Phi$, we write it in terms of the basis of simple roots as $\alpha=\sum_{i=1}^ra_i\alpha_i$ and define its \defn{height} $\h(\alpha):=\sum_{i=1}^ra_i$ as the sum of the coefficients.
Notice that $\h(\alpha)>0$ if and only if $\alpha\in\Phi^+$ and $\h(\alpha)=1$ if and only if $\alpha\in\simp$.
The highest root $\hr$ is the unique root in $\Phi$ of maximal height. Its coefficients in terms of the basis of simple roots can be found for example in \cite[Section 4.9]{humphreys90reflection}.
We define the \defn{Coxeter number} of $\Phi$ as $h:=1+\h(\hr)$.
\begin{mdframed}[style=Example]
\textbf{Example.} The highest root of the root system of type $A_{n-1}$ is $\hr=\alpha_1+\alpha_2+\ldots+\alpha_{n-1}$. It has height $n-1$.
So the Coxeter number $h$ of the root system of type $A_{n-1}$ equals $n$.
\end{mdframed}
For any integer $t$, write $\Phi_t$ for the set of roots of height $t$. Then in particular we have $\Phi_1=\simp$ and $\Phi_{h-1}=\{\hr\}$.\\
\\
Define the \defn{height} of an affine root $\alpha+k\delta$ as $\h(\alpha+k\delta)=\h(\alpha)+kh$.
So $\h(\alpha+k\delta)>0$ if and only if $\alpha+k\delta\in\ar^+$ and $\h(\alpha+k\delta)=1$ if and only if $\alpha+k\delta\in\widetilde{\simp}$.
For any integer $t$, write $\ar_t$ for the set of affine roots of height $t$.
\section{Classical rational Catalan combinatorics}\label{classical}
We aim to generalise the Anderson map $\A_{GMV}$ of Gorsky, Mazin and Vazirani \cite{gorsky14affine} to all irreducible crystallographic root systems.
This can be seen as a step towards a uniform theory of rational Catalan combinatorics. We start by expounding the theory of \defn{classical} rational Catalan combinatorics associated with type $A_{n-1}$.
\subsection{Rational Catalan numbers and rational Dyck paths}
For a positive integer $n$ and another positive integer $\rat$ relatively prime to $n$, the \defn{rational $(\rat,n)$-Catalan number} is defined as
\[\mathsf{Cat}_{\rat/n}:=\frac{1}{n+\rat}\binom{n+\rat}{n}.\]
These are generalisations of the classical Fuß-Catalan numbers: for a positive integer $m$, the rational Catalan number $\mathsf{Cat}_{(mn+1)/n}$ equals the classical Fuß-Catalan number $\mathsf{Cat}_n^{(m)}$.
It was proven by Bizley \cite{bizley54derivation} that $\mathsf{Cat}_{\rat/n}$ counts the number of \defn{rational $\rat/n$-Dyck paths}.
These are lattice paths in $\Z^2$ consisting of North and East steps that go from $(0,0)$ to $(\rat,n)$ and never go below the diagonal $y=\frac{n}{\rat}x$ of rational slope.\\
\\
For a $(\rat,n)$-Dyck path $P$ and $i\in[n]$, let $P_i$ be the $x$-coordinate of the $i$-th North step of $P$.
We may identify the path $P$ with the (weakly) increasing tuple of nonnegative integers $(P_1,P_2,\ldots,P_n)$. The condition that $P$ lies above the diagonal $y=\frac{n}{\rat}x$ translates to either 
\[P_i\leq\frac{\rat}{n}(i-1)\]
for all $i\in[n]$, or equivalently
\begin{equation}\label{ratdyck}
 \#\{i:P_i<l\}\geq\frac{nl}{\rat}
\end{equation}
\noindent
for all $l\in[\rat]$.\\
\\
The full lattice squares (boxes) between a $\rat/n$-Dyck path $P$ and the diagonal $y=\frac{n}{\rat}x$ we call \defn{area squares}.
The number of them in the $i$-th row from the bottom is the \defn{area} $a_i:=\lfloor \frac{\rat}{n}(i-1)\rfloor-P_i$ of that row.
We call the tuple $(a_1,a_2,\ldots,a_n)$ the \defn{area vector} of $P$.\\
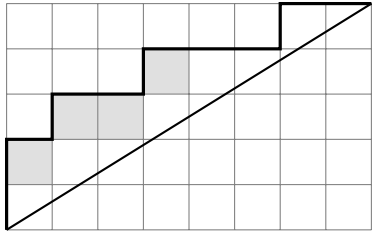
\begin{figure}[h]
\begin{center}
\begin{tikzpicture}[scale=.6]
\begin{scope}
	\fill[black!12] (0,1)--(1,1)--(1,2)--(0,2)--cycle;
	\fill[black!12] (1,2)--(3,2)--(3,3)--(1,3)--cycle;
	\fill[black!12] (3,3)--(4,3)--(4,4)--(3,4)--cycle;
	\draw[gray] (0,0) grid (8,5);
	\draw[very thick] (0,0)--(0,2)--(1,2)--(1,3)--(3,3)--(3,4)--(6,4)--(6,5)--(8,5);
	\draw[thick] (0,0)--(8,5);
\end{scope}
\end{tikzpicture}
\caption{The rational $8/5$-Dyck path $P=(0,0,1,3,6)$ with its area squares marked in gray. It has area vector $(0,1,2,1,0)$.}
\end{center}
\end{figure}
\\
Rational $\rat/n$-Dyck paths were used by Anderson \cite{anderson02core}, who provided a bijection between them and a certain set of integer partitions called \defn{$(\rat,n)$-cores}. The Anderson map $\A_{GMV}$, which will be defined later, may be seen as an extension of that bijection.
\subsection{Rational parking functions}
Equation (\ref{ratdyck}) suggests the following generalisation of rational Dyck paths. A \defn{rational $\rat/n$-parking function} is any tuple $(f_1,f_2,\ldots,f_n)$ of $n$ nonnegative integers such that 
\[\#\{i:f_i<l\}\geq\frac{nl}{\rat}\]
for all $l\in[\rat]$. Thus rational $\rat/n$-Dyck paths correspond to \defn{increasing} rational $\rat/n$-parking functions.
\begin{mdframed}[style=Example]
\textbf{Example.} The tuple $f=(6,0,1,0,3)$ is a rational $8/5$-parking function. Its increasing rearrangement is the rational $8/5$-Dyck path $(0,0,1,3,6)$.
\end{mdframed}
An important property of rational parking functions is the following folklore theorem. Let $\packrat$ be the set of $\rat/n$-parking functions.
\begin{theorem}\label{parkreps}
  $\packrat$ is a set of representatives for the cosets of the cyclic subgroup generated by $(1,1,\ldots,1)$ in the abelian group $\Z_{\rat}^n$.
  \begin{proof}
   We represent elements of $\Z_{\rat}^n$ as tuples $f=(f_1,f_2,\ldots,f_n)$ with $f_i\in\{0,1,\ldots,\rat-1\}$ for all $i\in[n]$.
   Define 
   \[\summed(f):=\sum_{i=1}^nf_i\in\N.\]
   We claim that every coset $H$ of the cyclic subgroup generated by $(1,1,\ldots,1)$ in $\Z_{\rat}^n$ has a unique representative with minimal $\summed$.
     To see this, we prove the stronger statement that all representatives of $H$ have different $\summed$.
     Suppose that $f\in H$, $l\in[\rat]$ and 
     \[\summed\left(f-l(1,1,\ldots,1)\right)=\summed(f).\]
     We calculate that 
     \begin{equation}\label{sum=sum}
      \summed\left(f-l(1,1,\ldots,1)\right)=\summed(f)-nl+\rat\cdot\#\{i\in[n]:f_i-l<0\},
     \end{equation}
     \noindent
     so $-nl+\rat\cdot\#\{i\in[n]:f_i-l<0\}=0$. Thus $\rat$ divides $nl$.
     Since $\rat$ is relatively prime to $n$, this implies that $\rat$ divides $l$. Therefore $f-l(1,1,\ldots,1)=f$.\\
     \\
     To finish the proof, we claim that $f\in H$ has minimal $\summed$ if and only if $f$ is a $\rat/n$-parking function.
     To see this, first suppose that $f\in H$ has minimal $\summed$. Then $\summed\left(f-l(1,1,\ldots,1)\right)\geq\summed(f)$ for all $l\in[\rat]$.
     From Equation (\ref{sum=sum}) we deduce that $-nl+\rat\cdot\#\{i\in[n]:f_i-l<0\}\geq0$ for all $l\in[\rat]$, or equivalently $\#\{i\in[n]:f_i<l\}\geq\frac{nl}{\rat}$ for all $l\in[\rat]$.
     So $f$ is a $\rat/n$-parking function. Reversing the argument shows that if $f$ is a $\rat/n$-parking function then it has minimal $\summed$ in the coset containing it.
  \end{proof}

\end{theorem}
\begin{corollary}[\protect{\cite[Corollary 4]{armstrong14rational}}]\label{parknum}
 $|\packrat|=\rat^{n-1}$.
\end{corollary}
\noindent
The symmetric group $S_n$ naturally acts on $\packrat$ by permuting coordinates. Every orbit of this action contains exactly one increasing $\rat/n$-parking function.
Thus the $S_n$-orbits on $\packrat$ are naturally indexed by $\rat/n$-Dyck paths. In particular the number of $S_n$-orbits on $\packrat$ is $\mathsf{Cat}_{\rat/n}$.
\subsection{Vertically labelled Dyck paths}\label{vert}
It is natural to introduce a combinatorial model for $\rat/n$-parking functions in terms of \defn{vertically labelled $\rat/n$-Dyck paths}.
In fact, this is how rational parking functions were originally defined in \cite{armstrong14rational}.\\
\\
An index $i\in[n]$ is called a \defn{rise} of a $\rat/n$-Dyck path $P$ if the $i$-th North step of $P$ is followed by another North step.
Equivalently, $i$ is a rise of $P$ if $P_i=P_{i+1}$.
A vertically labelled $\rat/n$-Dyck path is a pair $(P,\sigma)$ of a $\rat/n$-Dyck path $P$ and a permutation $\sigma\in S_n$ such that whenever $i$ is a rise of $P$ we have $\sigma(i)<\sigma(i+1)$.
We think of $\sigma$ as labeling the North steps of $P$ and say that the rise $i$ is \defn{labelled} $(\sigma(i),\sigma(i+1))$.\\
\\
The bijection from $\rat/n$-parking functions to vertically labelled $\rat/n$-Dyck paths works as follows:
for a $\rat/n$-parking function $f=(f_1,f_2,\ldots,f_n)$, let $(P_1,P_2,\ldots,P_n)$ be its increasing rearrangement and let $P$ be the corresponding $\rat/n$-Dyck path.
So $P$ encodes the values that $f$ takes, with multiplicities. In order to recover $f$, we also need to know their preimages under $f$.
Thus for all $l\in\{0,1,\ldots,\rat\}$ we label the North steps of $P$ with $x$-coordinate equal to $l$ by the preimages of $l$ under $f$.
If there are multiple North steps with the same $x$-coordinate we label them increasingly from bottom to top.
Let $\sigma\in S_n$ be the permutation that maps $i\in[n]$ to the label of the $i$-th North step of $P$.
Then $(P,\sigma)$ is the vertically labelled $\rat/n$-Dyck path corresponding to $f$.\\
\\
The inverse bijection is simple: if $(P,\sigma)$ is the vertically labelled $\rat/n$-Dyck path corresponding to the $\rat/n$-parking function $f$ then $f=\sigma\cdot(P_1,P_2,\ldots,P_n)$.
We will often use this bijection implicitly and do not distinguish between rational $\rat/n$-parking functions and their combinatorial interpretation as vertically labelled $\rat/n$-Dyck paths.\\
\begin{figure}[h]
\begin{center}
\begin{tikzpicture}[scale=.6]
\begin{scope}
	\draw[gray] (0,0) grid (8,5);
	\draw[very thick] (0,0)--(0,2)--(1,2)--(1,3)--(3,3)--(3,4)--(6,4)--(6,5)--(8,5);
	\draw[thick] (0,0)--(8,5);
	\draw[xshift=5mm,yshift=5mm]
		(-1,0) node[color=blue,]{\large{$2$}}
		(-1,1) node[color=blue,]{\large{$4$}}
		(0,2) node[color=blue,]{\large{$3$}}
		(2,3) node[color=blue,]{\large{$5$}}
		(5,4) node[color=blue,]{\large{$1$}};
\end{scope}
\end{tikzpicture}
\caption{The vertically labelled $8/5$-Dyck path $(P,\sigma)$ with $P=(0,0,1,3,6)$ and $\sigma=24351$.
It corresponds to the $8/5$-parking function $(6,0,1,0,3)$. Its only rise is $1$, and it is labelled $(2,4)$.}
\end{center}
\end{figure}
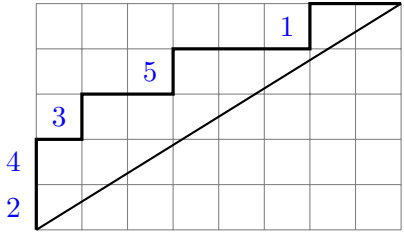
\\
In terms of vertically labelled Dyck paths, the natural $S_n$-action on $\pf_{\rat/n}$ is realized by defining for $\tau\in S_n$
$$\tau\cdot(P,\sigma)=(P,(\tau\sigma)'),$$
where $(P,(\tau\sigma)')$ comes from labelling the North steps of $P$ with $\tau\sigma$ and then sorting the labels in each column increasingly from bottom to top.\\
\section{The affine symmetric group}
Let $\Phi$ be a root system of type $A_{n-1}$. We choose
\begin{align*}
\Phi&=\{e_i-e_j:i,j\in[n],i\neq j\},\\
\Phi^+&=\{e_i-e_j:i,j\in[n],i<j\},\\
\Delta&=\{e_i-e_{i+1}:i\in[n-1]\},\\
V&=\{(x_1,x_2,\ldots,x_n)\in\R^n:\sum_{i=1}^nx_i=0\},\text{ and}\\
\Q&=\{(x_1,x_2,\ldots,x_n)\in\Z^n:\sum_{i=1}^nx_i=0\}.
\end{align*}
The Weyl group $W$ is the symmetric group $S_n$ acting on $V$ by permuting coordinates, the rank of $\Phi$ is $r=n-1$ and the Coxeter number is $h=n$.\\
\\
The affine Weyl group $\wa$ also has a combinatorial model as $\widetilde{S}_n$, the set of \defn{affine permutations} of period $n$. These are the bijections $\tilde{\sigma}:\Z\rightarrow\Z$ with 
$$\tilde{\sigma}(i+n)=\tilde{\sigma}(i)+n\text{ for all }i\in\Z\text{ and}$$
$$\sum_{i=1}^n\tilde{\sigma}(i)=\binom{n+1}{2}.$$
The affine symmetric group is generated by the \defn{affine simple transpositions} $\widetilde{s}_0,\widetilde{s}_1,\ldots,\widetilde{s}_{n-1}$ that act on $\Z$ by
\begin{align*}
 \widetilde{s}_j(i)&=i+1\text{ for }i\equiv j\text{ }(\mathrm{mod} \text{ }n),\\
 \widetilde{s}_j(i)&=i-1\text{ for }i\equiv j+1\text{ }(\mathrm{mod} \text{ }n),\text{ and}\\
 \widetilde{s}_j(i)&=i\text{ otherwise.}
\end{align*}
To identify the affine Weyl group $\wa$ with $\widetilde{S}_n$ we index its generating set as $\widetilde{S}=\{s_0,s_1,\ldots,s_{n-1}\}$, where $s_i=s_{e_i-e_{i+1}}$ for $i\in[n-1]$ and $s_0=s_{e_1-e_n}^1$.
Here $e_1-e_n=\hr$ is the highest root of $\Phi$. The generators $s_0,s_1,\ldots,s_{n-1}$ of $\wa$ satisfy the same relations as the generators $\widetilde{s}_0,\widetilde{s}_1,\ldots,\widetilde{s}_{n-1}$ of $\widetilde{S}_n$,
so sending $s_i\mapsto\widetilde{s}_i$ for $i=0,1\ldots,n-1$ defines an isomorphism from $\wa$ to $\widetilde{S}_n$.
We use this identification extensively and do not distinguish between elements of the affine Weyl group and the corresponding affine permutations.\\
\\
Since $\waf\in\widetilde{S}_n$ is uniquely defined by its values on $[n]$, we sometimes write it in \defn{window notation} as $\waf=[\waf(1),\waf(2),\ldots,\waf(n)]$.\\
\\
For $\waf\in\widetilde{S}_n$, write $\waf(i)=w(i)+n\mu_i$ with $w(i)\in[n]$ for all $i\in[n]$. 
Then $w\in S_n$, $\mu:=(\mu_1,\mu_2,\ldots,\mu_n)\in\Q$ and $\waf=wt_{\mu}\in\wa$.\\
\begin{mdframed}[style=Example]
\textbf{Example.} Consider the affine permutation $\waf\in\widetilde{S}_4$ with window $[-3,10,4,-1]$.
We can write $(-3,10,4,-1)=(1,2,4,3)+4(-1,2,0,-1)$, so we have $w=1243$ as a permutation in $S_4$ in one-line notation and $\mu=(-1,2,0,-1)\in\Q$. 
So $\waf=wt_{\mu}=s_3t_{(-1,2,0,-1)}\in\wa$.
\end{mdframed}
In order to avoid notational confusion, we use a ``$\cdot$'' for the action of $\wa$ on $V$.
So in particular $\waf(0)$ is the affine permutation $\waf$ evaluated at $0\in\Z$, whereas $\waf\cdot0$ is the affine Weyl group element $\waf$ applied to $0\in V$.
\section{Abaci}\label{abacus}
For any affine permutation $\waf$, we consider the set 
$$\Delta_{\waf}:=\{l\in\Z:\waf(l)>0\}=\waf^{-1}(\Z_{>0}).$$
We define its \defn{abacus diagram} $\mathsf{A}(\Delta_{\waf})$ as follows: draw $n$ \defn{runners}, labelled $1,2,\ldots,n$ from left to right, with runner $i$ containing all the integers congruent to $i$ modulo $n$ arranged in increasing order from top to bottom.
We say that the $k$-th \defn{level} of the abacus contains the integers $(k-1)n+i$ for $i\in[n]$ and arrange the runners in such a way that the integers of the same level are on the same horizontal line.\\
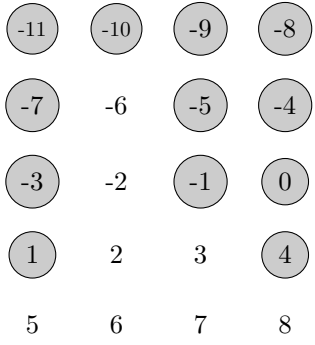
\begin{figure}[h]
\begin{center}
\begin{tikzpicture}[scale=1]
    \matrix[nodes={draw, circle, fill=gray!40},row sep=0.3cm,column sep=0.4cm] {
    \node[scale=0.8]{-11};&\node[scale=0.8]{-10};&\node{-9};&\node{-8};&\\
    \node{-7};&\node[draw=none, fill=none]{-6};&\node{-5};&\node{-4};&\\
    \node{-3};&\node[draw=none, fill=none]{-2};&\node{-1};&\node{0};&\\
    \node{1};&\node[draw=none, fill=none]{2};&\node[draw=none, fill=none]{3};&\node{4};&\\
    \node[draw=none, fill=none]{5};&\node[draw=none, fill=none]{6};&\node[draw=none, fill=none]{7};&\node[draw=none, fill=none]{8};&\\
    };
\end{tikzpicture}
\end{center}
\caption{The balanced $4$-flush abacus $\mathsf{A}(\Delta_{\waf})$ for $\waf=[-3,10,4,-1]$. Note that the values in the window of $\waf^{-1}=[5,-6,8,3]$ are the smallest (level) gaps of $\mathsf{A}(\Delta_{\waf})$ in each runner. The levels of the runners of $\mathsf{A}(\Delta_{\waf})$ are $1,-2,0,1$. Thus $\waf^{-1}\cdot0=(1,-2,0,1)$. Note that if we write $\waf=wt_{\mu}$ then $\waf^{-1}\cdot0=-\mu.$}
\end{figure}
\\
We circle the elements of $\Z\backslash\Delta_{\waf}$ and call them \defn{beads}, whereas we call the elements of $\Delta_{\waf}$ \defn{gaps}. Notice that the fact that $\waf(l+n)=\waf(l)+n>\waf(l)$ for all $l\in\Z$ implies that whenever $l\in\Delta_{\waf}$ then also $l+n\in\Delta_{\waf}$. 
We say that $\Delta_{\waf}$ is \defn{$n$-invariant}. Thus the abacus $\mathsf{A}(\Delta_{\waf})$ is \defn{$n$-flush}, that is whenever $l$ is a gap then all the $l+kn$ for $k\in\Z_{>0}$ below it are also gaps.
Or equivalently whenever $l$ is a bead then so are all the $l-kn$ for $k\in\Z_{>0}$ above it. \\
\\
For an $n$-flush abacus $\mathsf{A}$ define $\mathsf{level}_i(\mathsf{A})$ to be the largest level of a bead on runner $i$ in $\mathsf{A}$ for $i\in[n]$. Define the integer tuple
$$\lvl(\mathsf{A})=(\mathsf{level}_1(\mathsf{A}),\mathsf{level}_2(\mathsf{A}),\ldots,\mathsf{level}_n(\mathsf{A})).$$
The following theorem is well-known.
\begin{theorem}\label{abacuscoroot}
 For $\waf\in\widetilde{S}_n$, we have $\lvl(\mathsf{A}(\Delta_{\waf}))=\waf^{-1}\cdot0$.
 \begin{proof}
  Note that $\lvl(\mathsf{A}(\Delta_{e}))=0$ and
  \begin{align*}
   \lvl(\mathsf{A}(\Delta_{\waf s_j}))&=\lvl(\mathsf{A}((\waf s_j)^{-1}(\Z_{>0})))\\
   &=\lvl(\mathsf{A}(s_j\waf^{-1}(\Z_{>0})))\\
   &=\lvl(\mathsf{A}(s_j(\Delta_{\waf})))\\
   &=s_j\cdot\lvl(\mathsf{A}(\Delta_{\waf}))
  \end{align*}
  for $\waf\in\widetilde{S}_n$ and $j=0,1,\ldots,n-1$. Thus the result follows by induction on the length $l(\waf)$ of $\waf$.
 \end{proof}

\end{theorem}
\noindent
In particular $\lvl(\mathsf{A}(\Delta_{\waf}))\in\Q$, so the sum of the levels of $\mathsf{A}(\Delta_{\waf})$ is zero.
We call such an abacus \defn{balanced}.\\
\\
Let $M_{\waf}$ be the minimal element of $\Delta_{\waf}$ (that is, the smallest gap of $\mathsf{A}(\Delta_{\waf})$) and define $\widetilde{\Delta}_{\waf}=\Delta_{\waf}-M_{\waf}$. This is also an $n$-invariant set, so we form its $n$-flush abacus $\mathsf{A}(\widetilde{\Delta}_{\waf})$.
This is a \defn{normalized} abacus, that is its smallest gap is $0$.\\
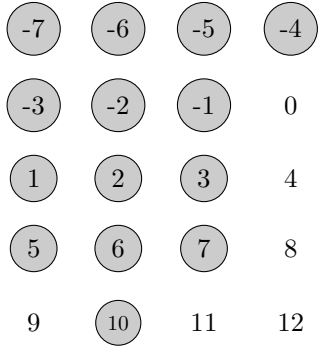
\begin{figure}[h]
\begin{center}
\begin{tikzpicture}[scale=1]
    \matrix[nodes={draw,circle,fill=gray!40},row sep=0.3cm,column sep=0.4cm] {
    \node{-7};&\node{-6};&\node{-5};&\node{-4};&\\
    \node{-3};&\node{-2};&\node{-1};&\node[draw=none, fill=none]{0};&\\
    \node{1};&\node{2};&\node{3};&\node[draw=none, fill=none]{4};&\\
    \node{5};&\node{6};&\node{7};&\node[draw=none, fill=none]{8};&\\
    \node[draw=none, fill=none]{9};&\node[scale=0.8]{10};&\node[draw=none, fill=none]{11};&\node[draw=none, fill=none]{12};&\\
    };
\end{tikzpicture}
\end{center}
\caption{The normalized $4$-flush abacus $\mathsf{A}(\widetilde{\Delta}_{\waf})$ for $\waf=[-3,10,4,-1]$. Here $M_{\waf}=-6$.}
\end{figure}
\\
\textbf{Remark}. It is easy to see that if $\Delta$ is an $n$-invariant set with $\lvl(\mathsf{A}(\Delta))=(x_1,x_2,\ldots,x_n)$, then $\lvl(\mathsf{A}(\Delta+1))=(x_n+1,x_1,x_2,\ldots,x_{n-1})$.
Thus we define the bijection 
\begin{align*}
 g:\Z^n&\rightarrow\Z^n\\
 (x_1,x_2,\ldots,x_n)&\mapsto (x_n+1,x_1,x_2,\ldots,x_{n-1})
\end{align*}
and get that
\begin{equation}\label{glvl}
 \lvl(\mathsf{A}(\widetilde{\Delta}_{\waf}))=g^{-M_{\waf}}\cdot\lvl(\mathsf{A}(\Delta_{\waf})).
\end{equation}
In particular $\sum_{i=1}^n\mathsf{level}_i(\mathsf{A}(\widetilde{\Delta}_{\waf}))=-M_{\waf}$.
\section{$\rat$-stable affine permutations}
For a positive integer $\rat$ relatively prime to $n$, we define the set of \defn{$\rat$-stable affine permutations} $\widetilde{S}_n^{\rat}$ as \cite[Definition 2.13]{gorsky14affine}
\[\widetilde{S}_n^{\rat}:=\{\waf\in \widetilde{S}_n:\waf(i+\rat)>\waf(i)\text{ for all }i\in\Z\}.\]
If $\waf$ is $\rat$-stable, then $\Delta_{\waf}$ is $\rat$-invariant in addition to being $n$-invariant.
So the $n$-flush abacus $\mathsf{A}(\Delta_{\waf})$ is also $\rat$-flush, that is whenever $l$ is a gap then so is $l+k\rat$ for all $k\in\Z_{>0}$.
\begin{mdframed}[style=Example]
\textbf{Example.} The affine permutation $\waf=[-3,10,4,-1]$ is $9$-stable. Thus its balanced abacus $\mathsf{A}(\Delta_{\waf})$ is $9$-flush in addition to being $4$-flush.
\end{mdframed}
\section{The combinatorial Anderson map}\label{andgmv}
We are now ready to describe the \defn{Anderson map} $\A_{GMV}$ defined by Gorsky, Mazin and Vazirani \cite[Section 3.1]{gorsky14affine}.
It is a bijection from the set of $\rat$-stable affine permutations $\widetilde{S}_n^{\rat}$ to the set of $\rat/n$-parking functions.\\
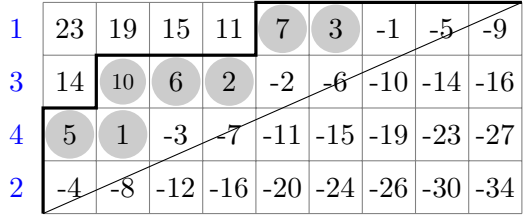
\begin{figure}[h]
\begin{center}
\begin{tikzpicture}[scale=.7]
\begin{scope}
	\draw[gray] (0,0) grid (9,4);
	\draw[very thick] (0,0)--(0,2)--(1,2)--(1,3)--(4,3)--(4,4)--(9,4);
	\draw (0,0)--(9,4);
	\draw[xshift=5mm,yshift=5mm]
		(0,0) node{\large{-4}}
		(1,0) node{\large{-8}}
		(2,0) node{\large{-12}}
		(3,0) node{\large{-16}}
		(4,0) node{\large{-20}}
		(5,0) node{\large{-24}}
		(6,0) node{\large{-26}}
		(7,0) node{\large{-30}}
		(8,0) node{\large{-34}}
		(0,1) node[fill=gray!40,circle]{\large{5}}
		(1,1) node[fill=gray!40,circle]{\large{1}}
		(2,1) node{\large{-3}}
		(3,1) node{\large{-7}}
		(4,1) node{\large{-11}}
		(5,1) node{\large{-15}}
		(6,1) node{\large{-19}}
		(7,1) node{\large{-23}}
		(8,1) node{\large{-27}}
		(0,2) node{\large{14}}
		(1,2) node[fill=gray!40,circle,scale=0.78]{\large{10}}
		(2,2) node[fill=gray!40,circle]{\large{6}}
		(3,2) node[fill=gray!40,circle]{\large{2}}
		(4,2) node{\large{-2}}
		(5,2) node{\large{-6}}
		(6,2) node{\large{-10}}
		(7,2) node{\large{-14}}
		(8,2) node{\large{-16}}
		(0,3) node{\large{23}}
		(1,3) node{\large{19}}
		(2,3) node{\large{15}}
		(3,3) node{\large{11}}
		(4,3) node[fill=gray!40,circle]{\large{7}}
		(5,3) node[fill=gray!40,circle]{\large{3}}
		(6,3) node{\large{-1}}
		(7,3) node{\large{-5}}
		(8,3) node{\large{-9}};
	\draw[xshift=5mm,yshift=5mm]
		(-1,0) node[color=blue,]{\large{2}}
		(-1,1) node[color=blue,]{\large{4}}
		(-1,2) node[color=blue,]{\large{3}}
		(-1,3) node[color=blue,]{\large{1}};
\end{scope}
\end{tikzpicture}
\caption{The vertically labelled $9/4$-Dyck path $\A_{GMV}(\waf)$ for $\waf=[-3,10,4,-1]$. It has area vector $(0,2,3,2)$ and labelling $\sigma=2431$. It corresponds to the $9/4$-parking function $(4,0,1,0)$.
The positive beads of the normalized abacus $\mathsf{A}(\widetilde{\Delta}_{\waf})$ are shaded in gray.}
\end{center}
\end{figure}
\\
We use English notation, so for us a \defn{Young diagram} is a finite set of square boxes that is left-justified and top-justified.
  Take $\waf\in\widetilde{S}^{\rat}_n$.
  As in Section \ref{abacus} we consider the set $$\Delta_{\waf}:=\{i\in\Z:\waf(i)>0\}$$ and let $M_{\waf}$ be its minimal element.
  Let $\widetilde{\Delta}_{\waf}:=\Delta_{\waf}-M_{\waf}$. In contrast to \cite{gorsky14affine}, we shall use $\widetilde{\Delta}_{\waf}$ in place of $\Delta_{\waf}$ and therefore also have a different labelling of $\Z^2$.\\
  \\
  View the integer lattice $\Z^2$ as the set of square boxes. Define the rectangle $$R_{\rat,n}:=\{(x,y)\in\Z^2:0\leq x<\rat,0\leq y<n\}$$ and label $\Z^2$ by the linear function
  $$l(x,y):=-n-nx+\rat y.$$
  Define the Young diagram 
  $$D_{\waf}:=\{(x,y)\in R_{\rat,n}:l(x,y)\in\widetilde{\Delta}_{\waf}\}$$
  and let $P_{\waf}$ be the path that defines its lower boundary. It is a $\rat/n$-Dyck path.
  Label its $i$-th North step by $\sigma(i):=\waf(l_i+M_{\waf})$, where $l_i$ is the label of the rightmost box of $D_{\waf}$ in the $i$-th row from the bottom (or the label of $(-1,i-1)$ if its $i$-th row is empty).
  Then $\sigma\in S_n$ and $(P_{\waf},\sigma)$ is a vertically labelled $\rat/n$-Dyck path.
  We define $\A_{GMV}:=(P_{\waf},\sigma)$.

\begin{theorem}[\protect{\cite[Theorem 3.4]{gorsky14affine}}]
 The Anderson map $\A_{GMV}$ is a bijection from $\widetilde{S}_n^{\rat}$ to the set of vertically labelled $\rat/n$-Dyck paths.
\end{theorem}
\section{The uniform Anderson map}\label{and}
In this section, we will generalise the Anderson map $\A_{GMV}$ to a \defn{uniform} Anderson map $\A$ that is defined for all irreducible crystallographic root systems $\Phi$.
It is a bijection from the set $\wa^{\rat}$ of \defn{$\rat$-stable} affine Weyl group elements to the \defn{finite torus} $\Q/\rat\Q$.
We will proceed in several steps, all of which have already appeared in the literature in some form. This section can be summarised in the following commutative diagram:
\begin{center}
 \begin{tikzpicture}
  \node (a) {$\wa^{\rat}$};
  \node[right=1.5cm of a] (b) {$\sommers$};
  \node[right=4cm of a] (c) {$\rat\ac$};
  \node[right=7cm of a] (d) {$\Q/\rat\Q$};
  \draw[->]
      (a) edge node (1) [above]{} (b)
      (b) edge node (2) [above]{} (c)
      (c) edge node (3) [above]{} (d)
      (a) edge[bend left] node (and) [above]{$\A$} (d);
\end{tikzpicture}
\end{center}
\subsection{$\rat$-stable affine Weyl group elements}\label{stable}
Let $\Phi$ be any irreducible crystallographic root system and let $\wa$ be its affine Weyl group.
We say that $\waf\in\wa$ is \defn{$\rat$-stable} if it has no inversions of height $\rat$.
That is, $\waf$ is $\rat$-stable if $\waf(\ar_{\rat})\subseteq\ar^+$.
We denote the set of $\rat$-stable affine Weyl group elements by $\wa^{\rat}$.\\
\\
Define 
\[^{\rat}\wa:=\{\waf^{-1}:\waf\in\wa^{\rat}\}\]
as the set of inverses of $\rat$-stable affine Weyl group elements. We call these elements \defn{$\rat$-restricted}.
Recall from Lemma \ref{arsep} that an affine root $\alpha+k\delta\in\ar^+$ is an inversion of $\waf\in\wa$ if and only if the corresponding hyperplane $H_{\alpha}^{-k}$ separates $\waf^{-1}\ac$ from $\ac$.
Thus $\waf\in{^{\rat}\wa}$ if and only if no hyperplane corresponding to an affine root of height $\rat$ separates $\waf\ac$ from $\ac$.\\
\\
Write $\rat=ah+b$, where $a$ and $b$ are nonnegative integers and $0<b<h$. Then we have
\begin{align*}
 \ar_{\rat}&=\{\alpha+a\delta:\alpha\in\Phi_b\}\cup\{\alpha+(a+1)\delta:\alpha\in\Phi_{b-h}\}\\
 &=\{\alpha+a\delta:\alpha\in\Phi_b\}\cup\{-\alpha+(a+1)\delta:\alpha\in\Phi_{h-b}\}
\end{align*}
Thus the hyperplanes corresponding to affine roots of height $\rat$ are those of the form $H_{\alpha}^{-a}$ with $\alpha\in\Phi_b$ and those of the form $H_{-\alpha}^{-(a+1)}=H_{\alpha}^{a+1}$ for $\alpha\in\Phi_{h-b}$.
We define the \defn{Sommers region} as the region in $V$ bounded by these hyperplanes:
\[\sommers:=\{x\in V:\langle x,\alpha\rangle>-a\text{ for all }\alpha\in\Phi_b\text{ and }\langle x,\alpha\rangle<a+1\text{ for all }\alpha\in\Phi_{h-b}\}.\]
We will soon see that this region is in fact a simplex. For now we make the following observation.
\begin{figure}[h]
\begin{center}
 \resizebox*{7cm}{!}{\includegraphics{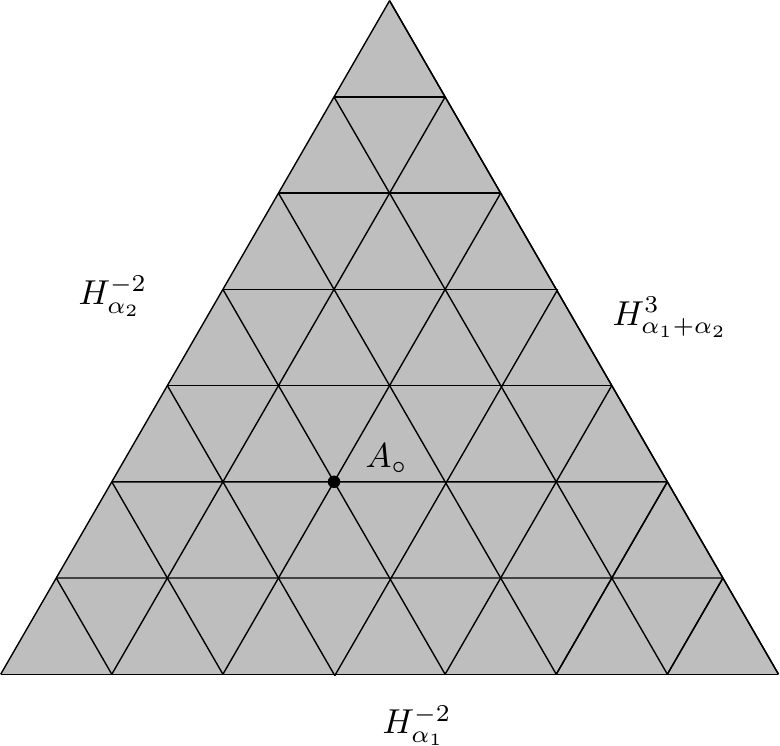}}\\
\end{center}
\caption{The $49$ alcoves in $\mathcal{S}_{\Phi}^7$ for $\Phi$ of type $A_2$.}
\end{figure}
\begin{lemma}\label{sommers}
 We have $\waf\in{^{\rat}\wa}$ if and only if $\waf\ac\subseteq\sommers$.
 \begin{proof}
  We have $\waf\in{^p\wa}$ if and only if no hyperplane corresponding to an affine root of height $\rat$ separates $\waf\ac$ from $\ac$.
  But those hyperplanes are exactly the hyperplanes bounding the Sommers region $\sommers$, and $\ac$ is inside $\sommers$.
 \end{proof}
\end{lemma}
\subsection{From the Sommers region to the dilated fundamental alcove}\label{sommtalc}
Fan has shown that there is an affine isometry that maps $\sommers$ to $\rat\ac$, the $\rat$-th dilation of the fundamental alcove \cite[Section 2.3]{fan96euler}.
Sommers has observed that we may choose this isometry to be an element of the affine Weyl group \cite[Theorem 5.7]{sommers05stable}.
\begin{figure}[h]
\begin{center}
 \resizebox*{7cm}{!}{\includegraphics{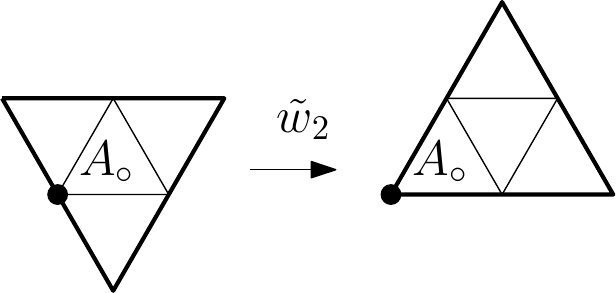}}\\
\end{center}
\caption{For the root system of type $A_2$, the affine Weyl group element $\waf_2=s_{\alpha_1+\alpha_2}^1$ maps the Sommers region $\mathcal{S}_{\Phi}^2$ to the dilated fundamental alcove $2\ac$.}
\end{figure}
\noindent
We will show that there is only one such element.\\
\\
Define $\{\cw_1,\cw_2,\ldots,\cw_r\}$ as the dual basis to the basis $\{\alpha_1,\alpha_2,\ldots,\alpha_r\}$ of simple roots.
Then $\cw_1,\cw_2,\ldots,\cw_r$ are called the \defn{fundamental coweights}. They generate the \defn{coweight lattice}
\[\cwl:=\{x\in V:\langle x,\alpha\rangle\in\Z\text{ for all }\alpha\in\Phi\}.\]
The coroot lattice $\Q$ is a sublattice of $\cwl$.
We define the \defn{index of connection} as
\[f:=[\cwl:\Q].\]
Write $\hr=\sum_{i=1}^rc_i\alpha_i$ as an integer linear combination of the simple roots.
\begin{lemma}\label{ll}
 Any prime that divides the index of connection $f$ or a coefficient $c_i$ for $i\in[r]$ also divides the Coxeter number $h$.
 \begin{proof}
  This may easily be checked \textbf{case-by-case} using tables such as \cite[Section 4.9]{humphreys90reflection}.
  We give a uniform proof also. We have that
  \[\frac{h^r}{c_1c_2\cdots c_rf}=\frac{r!h^r}{|W|}\in\Z\]
  The equality is a consequence of \defn{Weyl's formula} $|W|=r!c_1c_2\cdots c_rf$ \cite[Section 4.9]{humphreys90reflection}.
  The integrality of $\frac{r!h^r}{|W|}$ follows from the fact that it is the degree of the Lyashko-Looijenga morphism \cite[Theorem 5.3]{bessis07finite}.
  Thus any prime that divides $c_1c_2\cdots c_rf$ divides $h$ also.
 \end{proof}

\end{lemma}

\begin{theorem}\label{wf}
 There exists a unique $\wf\in\wa$ with $\wf(\sommers)=\rat\ac$.
 \begin{proof}
  It remains to show that if $\waf\in\wa$ and $\waf(\rat\ac)=\rat\ac$ then $\waf=e$ is the identity.\\
  \\
  Write the highest root as $\hr=\sum_{i=1}^rc_i\alpha_i$.
  The fundamental alcove $\ac$ has a vertex at $0$, and its other vertices are $\frac{1}{c_i}\cw_i$ for $i\in[r]$. Define $L$ as the lattice generated by $\{\frac{1}{c_i}\cw_i:i\in[r]\}$.
  Then 
  $$[L:\Q]=[L:\cwl][\cwl:\Q]=c_1c_2\cdots c_rf.$$
  So since $\rat$ is relatively prime to $h$ it is also relatively prime to $[L:\Q]$ by Lemma \ref{ll}. This implies that the map 
  \begin{align*}
   L/\Q&\rightarrow L/\Q\\
   x+\Q&\mapsto \rat x+\Q
  \end{align*}
  is invertible. Since $0$ is the only vertex of $\ac$ in the coroot lattice, $\frac{1}{c_i}\cw_i\notin\Q$ for all $i\in[r]$, so we also have $\rat\frac{1}{c_i}\cw_i\notin\Q$ for all $i\in[r]$.
  Thus $0$ is the only vertex of $\rat\ac$ that is in $\Q$.\\
  \\
  Now if $\waf(\rat\ac)=\rat\ac$, then $\waf\cdot0\in\Q$ must be a vertex of $\rat\ac$. Thus $\waf\cdot0=0$. So $\waf\in W$. Since $\waf(\rat\ac)=\rat\ac$ we must have $\waf C=C$, therefore $\waf=e$ as required.
 \end{proof}

\end{theorem}
\noindent
For a more explicit construction of $\wf$ see \cite[Section 4.3]{thiel15strange}.
\begin{corollary}\label{ratr}
 $|\wa^{\rat}|=|{^{\rat}\wa}|=\rat^r$.
 \begin{proof}
  By Lemma \ref{sommers} $|\wa^{\rat}|=|{^{\rat}\wa}|$ is the number of alcoves in the Sommers region $\sommers$.
  By Theorem \ref{wf} this equals the number of alcoves in $\rat\ac$. But the volume of $\rat\ac$ is $\rat^r$ times that of $\ac$, so it contains $\rat^r$ alcoves.
 \end{proof}

\end{corollary}

\subsection{From the dilated fundamental alcove to the finite torus}\label{att}
We follow a remark in \cite[Section 6]{sommers97family}. Define the $\rat$-dilated affine Weyl group as 
\[\wa_{\rat}:=W\ltimes\rat\Q.\]
Let $I_{\rat}:=\{\waf\in\wa:\waf\ac\subseteq\rat\ac\}$.
Since $\overline{\ac}$ is a fundamental domain for the action of $\wa$ on $V$, $\rat\overline{\ac}$ is a fundamental domain for the action of $\wa_{\rat}$ on $V$.
Therefore $I_{\rat}$ is a set of right coset representatives of $\wa_{\rat}$ in $\wa$.\\
\\
Another set of right coset representatives of $\wa_{\rat}$ in $\wa$ is any set of (the translations corresponding to) representatives of the finite torus $\Q/\rat\Q$.
Thus we get a bijection from $I_{\rat}$ to (the translations corresponding to a set of representatives of) $\Q/\rat\Q$ by sending an element $\waf\in I_{\rat}$ to the translation that represents the same coset of $\wa_{\rat}$ in $\wa$.
Explicitly, for $\waf=wt_{\mu}$, this is given by $t_{\mu}=t_{-\waf^{-1}\cdot0}$. So the map
\begin{align*}
 I_{\rat}&\rightarrow\Q/\rat\Q\\
 \waf&\mapsto-\waf^{-1}\cdot0+\rat\Q
\end{align*}
is a bijection.
\begin{figure}[h]
\begin{center}
 \resizebox*{9cm}{!}{\includegraphics{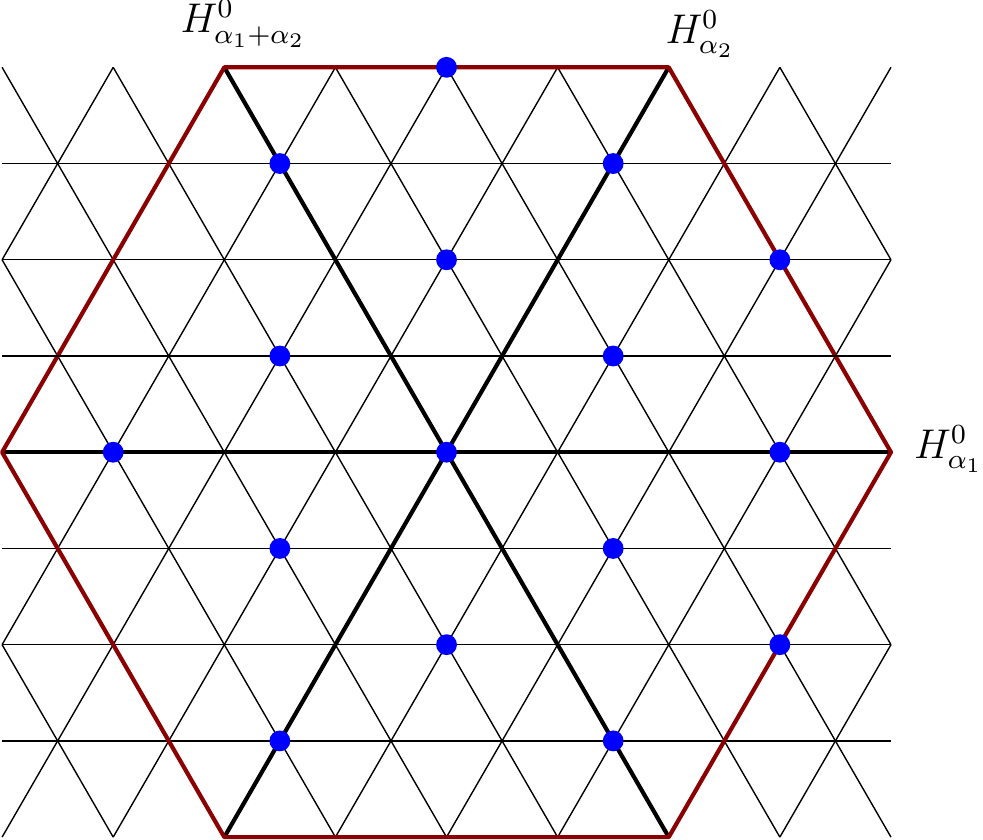}}\\
\end{center}
\caption{The blue dots are a natural set of representatives for the $16$ elements of the finite torus $\Q/4\Q$ of the root system of type $A_2$.}
\end{figure}
\subsection{Putting it all together}
We are now ready to define the uniform Anderson map $\A$ as
\begin{align*}
 \A:\wa^p&\rightarrow\Q/\rat\Q\\
 \waf&\mapsto \waf\waf_{\rat}^{-1}\cdot0+\rat\Q
\end{align*}
\begin{theorem}
 The Anderson map $\A$ is a bijection.
 \begin{proof}
  We start with $\waf\in\wa^{\rat}$. We take its inverse $\waf^{-1}\in{^{\rat}\wa}$. By Lemma \ref{sommers} we have $\waf^{-1}\ac\subseteq\sommers$.
  So by Theorem \ref{wf} we have $\waf_{\rat}\waf^{-1}\ac\subseteq\rat\ac$. That is $\waf_{\rat}\waf^{-1}\in I_{\rat}$.
  So as in Section \ref{att} we map it to 
  \[-(\waf_{\rat}\waf^{-1})^{-1}\cdot0+\rat\Q=-\waf\waf_{\rat}^{-1}\cdot0+\rat\Q\in\Q/\rat\Q.\]
  At the end we multiply by $-1$ to change sign. Each of the steps is bijective, so $\A$ is a bijection.
 \end{proof}

\end{theorem}
\subsection{The stabilizer of $\A(\waf)$}
At this point we prove a somewhat technical result about $\A$ that will be used in Section \ref{zetasec}.
The Weyl group $W$ acts on the coroot lattice $\Q$ and its dilation $\rat\Q$, so also on the finite torus $\Q/\rat\Q$.
For $\mu+\rat\Q\in\Q/\rat\Q$ we define its \defn{stabilizer} as
\[\mathsf{Stab}(\mu+\rat\Q):=\{w\in W:w(\mu+\rat\Q)=\mu+\rat\Q\}.\]
\begin{theorem}\label{stab}
 For $\waf\in\wa^{\rat}$ the stabilizer of $\A(\waf)\in\Q/\rat\Q$ is generated by $\{s_{\beta}:\beta\in\waf(\ar_{\rat})\cap\Phi\}$.
\end{theorem}
\noindent
To prove this, we will need the following result due to Haiman.
\begin{lemma}[\protect{ \cite[Lemma 7.4.1]{haiman94diagonal}}]\label{orb}
 The set $\rat\overline{\ac}\cap\Q$ is a system of representatives for the orbits of the $W$-action on $\Q/\rat\Q$. 
 The stabilizer of an element of $\Q/\rat\Q$ represented by $\mu\in\rat\overline{\ac}\cap\Q$ is generated by the reflections through the linear hyperplanes parallel to the walls of $\rat\overline{\ac}$ that contain $\mu$.
\end{lemma}
\begin{proof}[Proof of Theorem \ref{stab}]
 Suppose $\waf\in\wa^{\rat}$. Observe first that this implies $\waf(\ar_{\rat})\subseteq\ar^+$, so $\waf(\ar_{\rat})\cap\Phi\subseteq\Phi^+$.
 Write $\waf\waf_{\rat}^{-1}=ut_{-\mu}$.
 We have $\waf_{\rat}\waf^{-1}\ac\subseteq\rat\ac$, so $\mu=\waf_{\rat}\waf^{-1}\cdot0\in\rat\overline{\ac}\cap\Q$.
 We wish to show that the stabilizer of 
 \[\A(\waf)=\waf\waf_{\rat}^{-1}\cdot0+\rat\Q=-u(\mu)+\rat\Q\]
 in $\Q/\rat\Q$ is generated by $\{s_{\beta}:\beta\in\waf(\ar_{\rat})\cap\Phi\}$.\\
 \\
 First observe that $\waf_{\rat}$ maps the walls of $\sommers$ to the walls of $\rat\ac$.
 In terms of affine roots this means that 
 \[\waf_{\rat}(\ar_{\rat})=\simp\cup\{-\hr+\rat\delta\}.\]
 Now calculate that for $\beta\in\Phi^+$ we have the following equivalences:
 \begin{align*}
  &\beta\in\waf(\ar_{\rat})\\
  &\Leftrightarrow \waf^{-1}(\beta)\in\ar_{\rat}\\
  &\Leftrightarrow \waf_{\rat}\waf^{-1}(\beta)\in \waf_{\rat}(\ar_{\rat})=\simp\cup\{-\hr+\rat\delta\}\\
  &\Leftrightarrow \beta=\waf\waf_{\rat}^{-1}(\alpha)\text{ for some }\alpha\in\simp\text{ or }\beta=\waf\waf_{\rat}^{-1}(-\hr+\rat\delta)\\
  &\Leftrightarrow \beta=u(\alpha)\text{ and }\langle\mu,\alpha\rangle=0\text{ for some }\alpha\in\simp\text{ or }\beta=u(-\hr)\text{ and }\langle\mu,\hr\rangle=\rat.
 \end{align*}
 Here we used $\waf\waf_{\rat}^{-1}=ut_{-\mu}$ and the definition of the action of $\wa$ on $\ar$.
 Combining this with Lemma \ref{orb} we get
 \begin{align*}
  &\mathsf{Stab}(\waf\waf_{\rat}^{-1}\cdot0+\rat\Q)\\
  &=\mathsf{Stab}(-u(\mu)+\rat\Q)\\
  &=u\mathsf{Stab}(\mu+\rat\Q)u^{-1}\\
  &=u\langle s_{\alpha}:\mu\text{ lies in a wall of }\rat\overline{\ac}\text{ orthogonal to }\alpha\rangle u^{-1}\\
  &=u\langle s_{u^{-1}(\beta)}:\beta\in \waf(\ar_{\rat})\cap\Phi\rangle u^{-1}\\
  &=\langle s_{\beta}:\beta\in \waf(\ar_{\rat})\cap\Phi\rangle,
 \end{align*}
 as required.
\end{proof}
\section{The combinatorial Anderson map and the uniform Anderson map}\label{and+and}
It remains to relate the uniform Anderson map $\A$ defined in Section \ref{and} to the combinatorial Anderson map $\A_{GMV}$ defined in Section \ref{andgmv}.
So for this section, let $\Phi$ be a root system of type $A_{n-1}$.\\
\\
First note that the set of $\rat$-stable affine Weyl group elements $\wa^{\rat}$ coincides with the set of $\rat$-stable affine permutations $\widetilde{S}_n^{\rat}$ \cite[Section 2.3]{gorsky14affine}.
It remains to relate the finite torus $\Q/\rat\Q$ to the set $\packrat$ of rational $\rat/n$-parking functions.
\subsection{Parking functions and the finite torus}\label{ptt}
We follow \cite[Section 5.1]{athanasiadis05refinement}.
First recall that
$$\Q=\{(x_1,x_2,\ldots,x_n)\in\Z^n:\sum_{i=1}^nx_i=0\}.$$
The natural projection 
$$\mathrm{mod}\text{ } \rat:\Q\rightarrow \{(x_1,x_2,\ldots,x_n)\in\Z_{\rat}^n:\sum_{i=1}^nx_i=0\}$$
has kernel $\rat\Q$.
Futhermore, since $n$ and $\rat$ are relatively prime, the natural projection
$$\mathrm{mod}\text{ } (1,1,\ldots,1):\{(x_1,x_2,\ldots,x_n)\in\Z_{\rat}^n:\sum_{i=1}^nx_i=0\}\rightarrow \Z_{\rat}^n/(1,1,\ldots,1)$$
is a bijection to the set $\Z_{\rat}^n/(1,1,\ldots,1)$ of cosets of the cyclic subgroup of $\Z_{\rat}^n$ generated by $(1,1,\ldots,1)$.
Thus if $\pi_{\Q}:=\mathrm{mod}\text{ } (1,1,\ldots,1)\circ\mathrm{mod}\text{ } \rat$, then
$$\pi_{\Q}:\Q/\rat\Q\rightarrow\Z_{\rat}^n/(1,1,\ldots,1)$$
is a well-defined bijection.\\
\\
Recall from Theorem \ref{parkreps} that the set of rational $\rat/n$-parking functions is a set of representatives for $\Z_{\rat}^n/(1,1,\ldots,1)$, so the natural projection
$$\pi_{\pf}:\pf_{\rat/n}\rightarrow\Z_{\rat}^n/(1,1,\ldots,1)$$
is a bijection.\\
\\
Note that $W=S_n$ naturally acts on $\Q/\rat\Q$, $\pf_{\rat/n}$ and $\Z_{\rat}^n/(1,1,\ldots,1)$ and that both $\pi_{\Q}$ and $\pi_{\pf}$ are isomorphisms with respect to these actions.
So we define $\ttpf:=\pi_{\pf}^{-1}\circ\pi_{\Q}$ as the natural $S_n$-isomorphism from $\Q/\rat\Q$ to $\packrat$.
\subsection{The Anderson maps are equivalent}
The following theorem interprets the combinatorial Anderson map $\A_{GMV}$ as the uniform Anderson map $\A$ specialised to type $A_{n-1}$.
\begin{figure}[h]
\begin{center}
\begin{tikzpicture}
  \matrix (m) [matrix of math nodes,row sep=2em,column sep=5em,minimum width=2em]
  {
     {} & \Q/\rat\Q & {}\\
     \wa^{\rat}=\widetilde{S}_n^{\rat} & {} & \Z^n_{\rat}/(1,1,\ldots,1)\\
     {} & \packrat & {}\\};
  \path[-stealth]
    (m-2-1) edge node [above] {$\A$} (m-1-2)
    (m-3-2) edge node [above] {$\pi_{\pf}$} (m-2-3)
    (m-2-1) edge node [above] {$\A_{GMV}$} (m-3-2)
    (m-1-2) edge node [above] {$\pi_{\Q}$} (m-2-3)
    (m-1-2) edge node [left] {$\ttpf$} (m-3-2);
\end{tikzpicture}
\caption{Theorem \ref{GMV} as a commutative diagram of bijections.}
\end{center}
\end{figure}
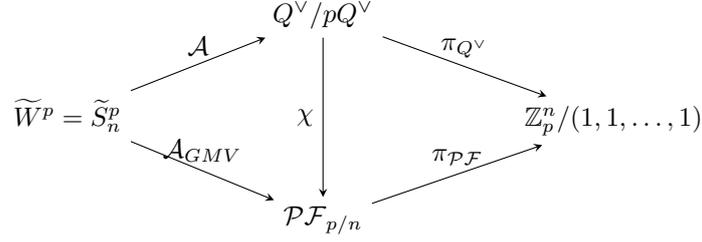
\begin{theorem}\label{GMV}
 Suppose $\Phi$ is of type $A_{n-1}$ and $\rat$ is a positive integer relatively prime to $n$. Then
 $$\pi_{\Q}\circ\A=\pi_{\pf}\circ\A_{GMV}.$$
 \begin{proof}
Let $\waf\in\wa^{\rat}=\widetilde{S}_n^{\rat}$. Refer to Section \ref{andgmv} for the construction of $\A_{GMV}(\waf)$. We employ the same notation as in that section here.\\
\\
We first consider the case where $\waf=\waf_{\rat}$, as defined in Theorem \ref{wf}.
Since 
\[\waf_{\rat}(\ar_{\rat})=\simp\cup\{-\hr+\rat\delta\}\subseteq\ar^+\]
we indeed have $\waf_{\rat}\in\wa^{\rat}$. From \cite[Lemma 2.16]{gorsky14affine} we get that its inverse is
$$\waf_{\rat}^{-1}=[\rat-c,2\rat-c,\ldots,n\rat-c]$$
where $c=\frac{(\rat-1)(n+1)}{2}$. Since $\Delta_{\waf_{\rat}}=\waf_{\rat}^{-1}(\Z_{>0})$ the set of lowest gaps of the runners of the balanced abacus $\mathsf{A}(\Delta_{\waf_{\rat}})$ is 
$$\{\waf_{\rat}^{-1}(1),\waf_{\rat}^{-1}(2)\ldots \waf_{\rat}^{-1}(n)\}=\{\rat-c,2\rat-c,\ldots,n\rat-c\}.$$
Thus the set of lowest gaps of the runners of the normalized abacus $\mathsf{A}(\widetilde{\Delta}_{\waf_{\rat}})$ is 
\[\{0,\rat,2\rat,\ldots,(n-1)\rat\}.\]
This is exactly the set of labels of $(-1,i-1)$ for $i\in[n]$. Thus all the labels in $R_{\rat,n}$ are beads in $\mathsf{A}(\widetilde{\Delta}_{\waf_{\rat}})$.
Therefore $D_{\waf_{\rat}}$ is empty and $\A_{GMV}(\waf_{\rat})=(P_{\waf_{\rat}},\sigma)=(0,0,\ldots,0)$.\\
\\
For $x,y\in\Z^n$ write $x\equiv y$ if the projections of $x$ and $y$ into $\Z_{\rat}^n/(1,1,\ldots,1)$ agree.
Then $\equiv$ is compatible both with addition and with the $S_n$-action on $\Z^n$. The set of lowest gaps of the runners of the abacus $\mathsf{A}(\widetilde{\Delta}_{\waf_{\rat}}+\rat)$ is $\{\rat,2\rat,\ldots,n\rat\}$. Thus
$$\lvl(\mathsf{A}(\widetilde{\Delta}_{\waf_{\rat}}+\rat))=\lvl(\mathsf{A}(\widetilde{\Delta}_{\waf_{\rat}}))+(0,0,\ldots,0,\rat)\equiv\lvl(\mathsf{A}(\widetilde{\Delta}_{\waf_{\rat}})).$$
We also have
$$\lvl(\mathsf{A}(\widetilde{\Delta}_{\waf_{\rat}}+n))=\lvl(\mathsf{A}(\widetilde{\Delta}_{\waf_{\rat}}))+(1,1,\ldots,1)\equiv\lvl(\mathsf{A}(\widetilde{\Delta}_{\waf_{\rat}})).$$
In terms of the bijection $g$ from Section \ref{abacus}, this means that
$$g^{\rat}\cdot\lvl(\mathsf{A}(\widetilde{\Delta}_{\waf_{\rat}}))\equiv\lvl(\mathsf{A}(\widetilde{\Delta}_{\waf_{\rat}}))$$
and
$$g^n\cdot\lvl(\mathsf{A}(\widetilde{\Delta}_{\waf_{\rat}}))\equiv\lvl(\mathsf{A}(\widetilde{\Delta}_{\waf_{\rat}})).$$
Since $\rat$ and $n$ are relatively prime, this implies that
\begin{equation}\label{g}
 g\cdot\lvl(\mathsf{A}(\widetilde{\Delta}_{\waf_{\rat}}))\equiv\lvl(\mathsf{A}(\widetilde{\Delta}_{\waf_{\rat}})).
\end{equation}
Now take any $\waf\in\widetilde{S}_n^{\rat}$. Note that the labels of boxes in the $i$-th row of the Young diagram $D_{\waf}$ from the bottom (those with $y$-coordinate $i-1$) are those congruent to $\rat(i-1)$ modulo $n$.
Thus we define the permutation $\tau\in S_n$ by 
$$\tau(i)\equiv \rat(i-1)\text{ }\mathrm{mod}\text{ }n$$
for all $i\in[n]$. The fact that $\rat$ is relatively prime to $n$ implies that this indeed gives a permutation of $n$.\\
\begin{figure}[h]
\begin{center}
\begin{tikzpicture}[scale=.7]
\begin{scope}
	\draw[gray] (0,0) grid (8,5);
	\draw[very thick] (0,0)--(0,2)--(1,2)--(1,3)--(3,3)--(3,4)--(6,4)--(6,5)--(8,5);
	\draw[thick] (0,0)--(8,5);
	\draw[xshift=5mm,yshift=5mm]
		(0,0) node{\large{-5}}
		(1,0) node{\large{-10}}
		(2,0) node{\large{-15}}
		(3,0) node{\large{-20}}
		(4,0) node{\large{-25}}
		(5,0) node{\large{-30}}
		(6,0) node{\large{-35}}
		(7,0) node{\large{-40}}
		(0,1) node[fill=gray!40,circle]{\large{3}}
		(1,1) node{\large{-2}}
		(2,1) node{\large{-7}}
		(3,1) node{\large{-12}}
		(4,1) node{\large{-17}}
		(5,1) node{\large{-22}}
		(6,1) node{\large{-27}}
		(7,1) node{\large{-32}}
		(0,2) node{\large{11}}
		(1,2) node[fill=gray!40,circle]{\large{6}}
		(2,2) node[fill=gray!40,circle]{\large{1}}
		(3,2) node{\large{-4}}
		(4,2) node{\large{-9}}
		(5,2) node{\large{-14}}
		(6,2) node{\large{-19}}
		(7,2) node{\large{-24}}
		(0,3) node{\large{19}}
		(1,3) node{\large{14}}
		(2,3) node{\large{9}}
		(3,3) node[fill=gray!40,circle]{\large{4}}
		(4,3) node{\large{-1}}
		(5,3) node{\large{-6}}
		(6,3) node{\large{-11}}
		(7,3) node{\large{-16}}
		(0,4) node{\large{27}}
		(1,4) node{\large{22}}
		(2,4) node{\large{17}}
		(3,4) node{\large{12}}
		(4,4) node{\large{7}}
		(5,4) node{\large{2}}
		(6,4) node{\large{-3}}
		(7,4) node{\large{-8}};
	\draw[xshift=5mm,yshift=5mm]
		(-1,0) node[color=blue,]{\large{2}}
		(-1,1) node[color=blue,]{\large{4}}
		(-1,2) node[color=blue,]{\large{3}}
		(-1,3) node[color=blue,]{\large{5}}
		(-1,4) node[color=blue,]{\large{1}};
\end{scope}
\end{tikzpicture}
\caption{The vertically labelled $8/5$-Dyck path $\A_{GMV}(\waf)$ for $\waf=[0,7,-2,6,4]$.
In this case we have $\tau=53142$ and $M_{\waf}=-3$.
The positive beads of the normalized abacus $\mathsf{A}(\widetilde{\Delta}_{\waf})$ are shaded in gray.
To reconstruct $\waf$ from $\A_{GMV}(\waf)$ one may proceed as follows: Read the labels of the boxes to the left of the North steps ($0,8,11,9,2$), reorder them according to the labels of the North steps ($2,0,11,8,9$) and add $M_{\waf}=1-\#\{\text{boxes between the path and the diagonal}\}$ to get the window notation of $\waf^{-1}=[-1,-3,8,5,6]$.
}
\end{center}
\end{figure}
\\
Let $P_i$ be the number of boxes on the $i$-th row of $D_{\waf}$ from the bottom.
This is the number of gaps of $\mathsf{A}(\widetilde{\Delta}_{\waf})$ on runner $\tau(i)$ that are in $R_{\rat,n}$.
Equivalently, it is the number of gaps of $\mathsf{A}(\widetilde{\Delta}_{\waf})$ on runner $\tau(i)$ that are smaller than the smallest gap on runner $\tau(i)$ of $\mathsf{A}(\widetilde{\Delta}_{\waf_{\rat}})$.
Thus 
$$P_i=\mathsf{level}_{\tau(i)}(\mathsf{A}(\widetilde{\Delta}_{\waf_{\rat}}))-\mathsf{level}_{\tau(i)}(\mathsf{A}(\widetilde{\Delta}_{\waf})),$$
that is
\begin{equation}\label{p}
 (P_1,P_2,\ldots,P_n)=\tau^{-1}\cdot\left[\lvl(\mathsf{A}(\widetilde{\Delta}_{\waf_{\rat}}))-\lvl(\mathsf{A}(\widetilde{\Delta}_{\waf}))\right].
\end{equation}
Now we start looking at the labelling $\sigma$ of the $\rat/n$-Dyck path $P_{\waf}$.
We have for $i\in[n]$
$$\sigma(i):=\waf(l_i+M_{\waf})\equiv \waf(\tau(i)+M_{\waf})\modn.$$
Define $r\in S_n$ by $r(i)\equiv i+1\modn$. Write $\waf=wt_{-\mu}$ with $w\in W=S_n$ and $\mu\in\Q$, simultaneously viewing $w$ as an affine permutation in $\widetilde{S}_n$ also.
Then 
$$w(r^{M_{\waf}}(\tau(i)))\equiv w(\tau(i)+M_{\waf})\equiv \waf(\tau(i)+M_{\waf})\equiv\sigma(i)\modn.$$
Since $\sigma(i)$ and $w(r^{M_{\waf}}(\tau(i)))$ are congruent modulo $n$ and both in $[n]$, they are equal.
Thus 
\begin{equation}\label{sigma}
 \sigma=w\circ r^{M_{\waf}}\circ\tau.
\end{equation}
Now we calculate
\begin{align*}
 \A_{GMV}(\waf)&=(P_{\waf},\sigma)\\
 &=\sigma\cdot(P_1,P_2,\ldots,P_n)\\
 &=(w\circ r^{M_{\waf}}\circ\tau)\cdot\tau^{-1}\cdot\left[\lvl(\mathsf{A}(\widetilde{\Delta}_{\waf_{\rat}}))-\lvl(\mathsf{A}(\widetilde{\Delta}_{\waf}))\right]\\
 &\equiv(w\circ r^{M_{\waf}})\cdot\left[g^{M_{\waf_{\rat}}-M_{\waf}}\cdot\lvl(\mathsf{A}(\widetilde{\Delta}_{\waf_{\rat}}))-\lvl(\mathsf{A}(\widetilde{\Delta}_{\waf}))\right]\\
 &=(w\circ r^{M_{\waf}})\cdot\left[g^{-M_{\waf}}\cdot\lvl(\mathsf{A}(\Delta_{\waf_{\rat}}))-g^{-M_{\waf}}\cdot\lvl(\mathsf{A}(\Delta_{\waf}))\right]\\
 &=(w\circ r^{M_{\waf}})\cdot\left[r^{-M_{\waf}}\cdot\left[\lvl(\mathsf{A}(\Delta_{\waf_{\rat}}))-\lvl(\mathsf{A}(\Delta_{\waf}))\right]\right]\\
 &=w\cdot(\waf_{\rat}^{-1}\cdot0-\waf^{-1}\cdot0)\\
 &=w\cdot(\waf_{\rat}^{-1}\cdot0-\mu)\\
 &=wt_{-\mu}\waf_{\rat}^{-1}\cdot0\\
 &=\waf \waf_{\rat}^{-1}\cdot0\\
 &\equiv\A(\waf).
\end{align*}
Here we used Equation (\ref{sigma}), Equation (\ref{p}), Equation (\ref{g}), Equation (\ref{glvl}) and Theorem \ref{abacuscoroot}, in that order.
\end{proof}

\end{theorem}
\section{Dominant $\rat$-stable affine Weyl group elements}
We say that $\waf\in\wa$ is \defn{dominant} if $\waf\ac$ is contained in the dominant chamber.
In this section, we will study the set $\wratdom$ of dominant $\rat$-stable affine Weyl group elements.
The following lemma is well-known.
\begin{lemma}\label{grass}
 $\waf\in\wa$ is dominant if and only if for all $w\in W$ we have $\inv(\waf)\subseteq\inv(w\waf)$.
 \begin{proof}
  Suppose $\waf\in\wa$ is dominant. Then no hyperplane of the (linear) Coxeter arrangement separates $\waf\ac$ from $\ac$.
  So by Lemma \ref{arsep} $\inv(\waf^{-1})\cap\Phi^+=\emptyset$. Equivalently 
  \[\waf(\ar^+)\cap-\Phi^+=\emptyset.\]
  So if $\alpha+k\delta\in\inv(\waf)$, then $\waf(\alpha+k\delta)\in-\ar^+\backslash(-\Phi^+)$. Thus $w\waf(\alpha+k\delta)\in-\ar^+$ and $\alpha+k\delta\in\inv(w\waf)$.\\
  \\
  Conversely, suppose $\waf\in\wa$ and $\inv(\waf)\subseteq\inv(w\waf)$ for all $w\in W$. Then for $\alpha+k\delta\in\inv(\waf)$ we have $w\waf(\alpha+k\delta)\in-\ar^+$ for all $w\in W$, so $\waf(\alpha+k\delta)\in-\ar^+\backslash(-\Phi^+)$.
  Thus 
  \[\waf(\ar^+)\cap-\Phi^+=\emptyset,\]
  or equivalently $\inv(\waf^{-1})\cap\Phi^+=\emptyset$. So by Lemma \ref{arsep} no hyperplane of the (linear) Coxeter arrangement separates $\waf\ac$ from $\ac$ and thus $\waf\ac$ is dominant.
 \end{proof}

\end{lemma}
\begin{mdframed}[style=Example]
\textbf{Example.} If $\Phi$ is of type $A_{n-1}$, so that $\wa=\widetilde{S}_n$, then the dominant affine Weyl group elements $\waf$ are exactly those whose inverse is \defn{affine Grassmanian}, that is those that satisfy
\[\waf^{-1}(1)<\waf^{-1}(2)<\cdots<\waf^{-1}(n).\]
\end{mdframed}
\begin{lemma}\label{ratdom}
 If $\waf\in\wa^{\rat}$ and $w\in W$ such that $\waf\ac\subseteq wC$, then $w^{-1}\waf\in\wratdom$.
 \begin{proof}
  Clearly $w^{-1}\waf$ is dominant. By Lemma \ref{grass} we have $\inv(w^{-1}\waf)\subseteq\inv(\waf)$, so 
  \[\inv(w^{-1}\waf)\cap\ar_{\rat}=\emptyset,\]
  thus $w^{-1}\waf\in\wa^{\rat}$.
 \end{proof}

\end{lemma}
\noindent
%
%
\section{The $\m$-Shi arrangement}\label{shiarr}
In this section, we introduce the \defn{$\m$-Shi arrangement} of an irreducible crystallographic root system $\Phi$, a hyperplane arrangement in $V$ whose study forms a major part of Fuß-Catalan combinatorics.
There is a close link between this section and the previous ones: The regions of the $\m$-Shi arrangement can be indexed by their \defn{minimal alcoves}, which correspond to the $(mh+1)$-stable affine Weyl group elements, the Fuß-Catalan specialisation ($\rat=mh+1$) of the set $\wa^{\rat}$ considered in Section \ref{and}.
\subsection{The Shi arrangement}\label{shi}
The \defn{Shi arrangement} is the affine hyperplane arrangement given by all the hyperplanes $H_{\alpha}^d$ for $\alpha\in\Phi^+$ and $d=0,1$.
It was first introduced in \cite{shi87sign} and arose from the study of the Kazhdan-Lusztig cells of the affine Weyl group of type $\widetilde{A}_{n-1}$.
The complement of these hyperplanes, we call them \defn{Shi hyperplanes}, falls apart into connected components which we call the \defn{regions} of the Shi arrangement, or Shi regions for short.
We call a Shi region \defn{dominant} if it is contained in the dominant chamber.
\begin{figure}[h]
\begin{center}
 \resizebox*{8cm}{!}{\includegraphics{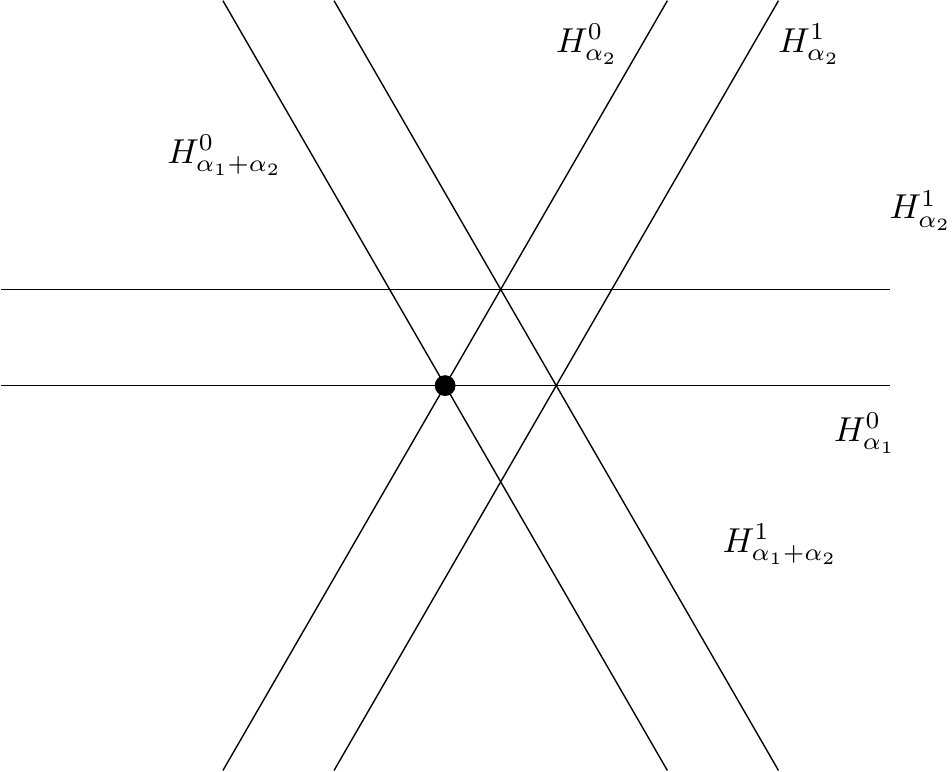}}\\
\end{center}
\caption{The Shi arrangement of the root system of type $A_2$. It has $16$ regions.}
\end{figure}
\\
An \defn{ideal} in the root poset is a subset $I\subseteq\Phi^+$ such that whenever $\alpha\in I$ and $\beta\leq\alpha$, then $\beta\in I$.
Dually, we define an \defn{order filter} as a subset $J\subseteq\Phi^+$ such that whenever $\alpha\in J$ and $\alpha\leq\beta$ then $\beta\in J$.
For a dominant Shi region $R$ define 
$$\phi(R):=\{\alpha\in\Phi^+:\langle x,\alpha\rangle>1\text{ for all }x\in R\}.$$
It is easy to see that $\phi(R)$ is an order filter in the root poset of $\Phi$.
In fact, $\phi$ even defines a bijection between the set of dominant Shi regions and the set of order filters in the root poset \cite[Theorem 1.4]{shi97number}.
\begin{figure}[h]
\begin{center}
 \resizebox*{7cm}{!}{\includegraphics{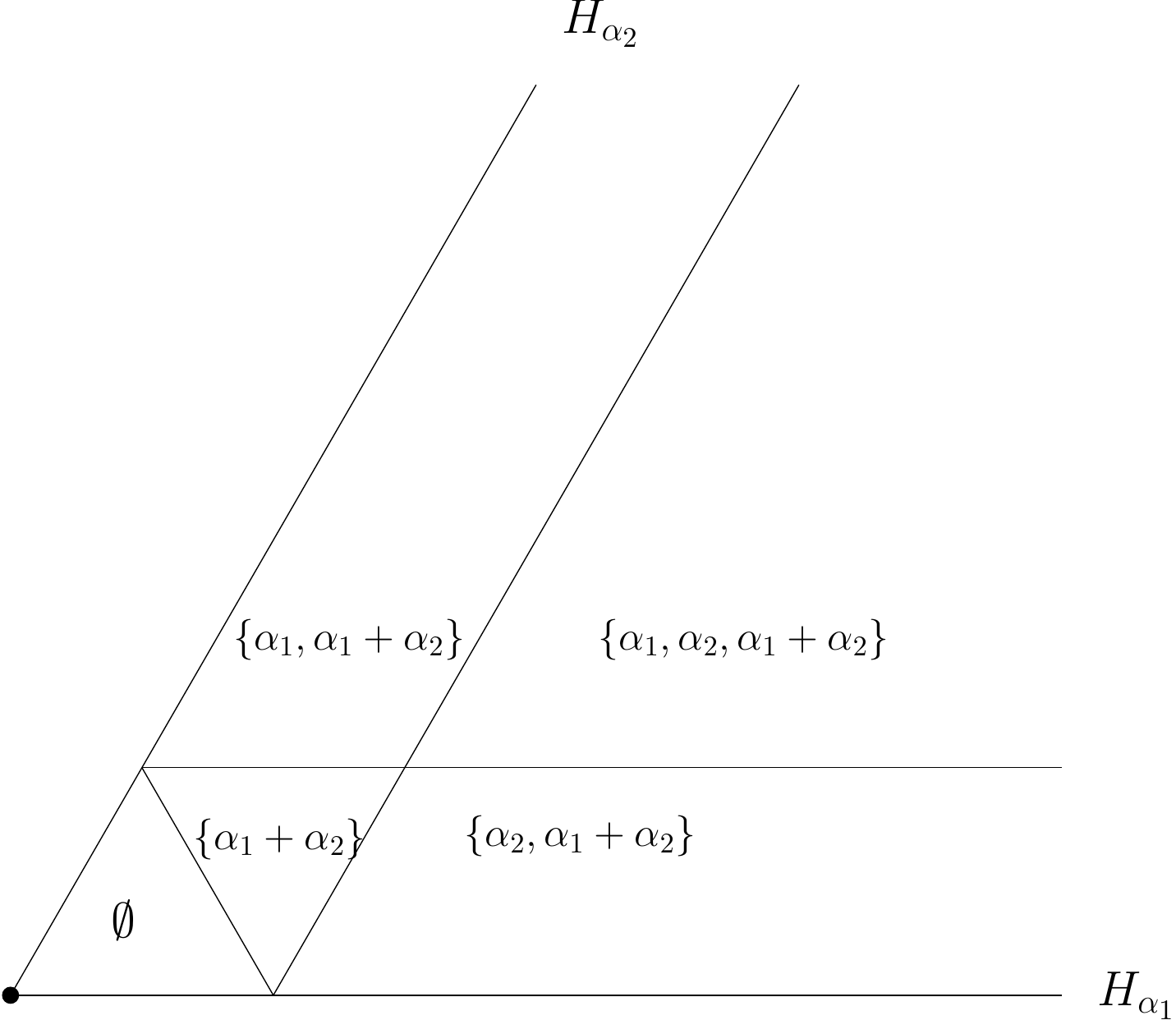}}\\
\end{center}
\caption{The $5$ dominant Shi regions of the root system of type $A_2$ together with their corresponding order filters in the root poset.}
\end{figure}
\subsection{The $\m$-extended Shi arrangement}\label{mshi}
For a positive integer $\m$, the \defn{$\m$-extended Shi arrangement}, or simply $\m$-Shi arrangement, is the affine hyperplane arrangement given by all the hyperplanes $H_{\alpha}^k$ for $\alpha\in\Phi^+$ and $-\m<k\leq \m$. We call them \defn{$\m$-Shi hyperplanes}.
The complement of these hyperplanes falls apart into connected components, which we call the \defn{regions} of the $\m$-Shi arrangement, or $\m$-Shi regions for short.
Notice that the $1$-Shi arrangement is exactly the Shi arrangement introduced in Section \ref{shi}.\\
\\
Following \cite{athanasiadis05refinement}, we will encode dominant $\m$-Shi regions by \defn{geometric chains of ideals} or equivalently \defn{geometric chains of order filters}.
Suppose $\I=(I_1,I_2,\ldots,I_{\m})$ is an ascending (multi)chain of $\m$ ideals in the root poset of $\Phi$, that is $I_1\subseteq I_2\subseteq\ldots\subseteq I_{\m}$.
Setting $J_i:=\Phi^+\backslash I_i$ for $i\in[\m]$ and $\J:=(J_1,J_2,\ldots,J_{\m})$ gives us the corresponding descending chain of order filters. That is, we have $J_1\supseteq J_2\supseteq\ldots\supseteq J_{\m}$.
The ascending chain of ideals $\I$ and the corresponding descending chain of order filters $\J$ are both called \defn{geometric} if the following conditions are satisfied simultaneously.
\begin{enumerate}
 \item $(I_i+I_j)\cap\Phi^+\subseteq I_{i+j}\text{ for all }i,j\in\{0,1,\ldots,\m\}\text{ with }i+j\leq \m\text{, and}$
 \item $(J_i+J_j)\cap\Phi^+\subseteq J_{i+j}\text{ for all }i,j\in\{0,1,\ldots,\m\}\text{.}$
\end{enumerate}
Here we set $I_0:=\varnothing$, $J_0:=\Phi^+$ and $J_i:=J_{\m}$ for $i>\m$.\\
\begin{mdframed}[style=Example]
\textbf{Example.} For example, the chain of $2$ order filters $\J=(\{\alpha_1+\alpha_2\},\{\alpha_1+\alpha_2\})$ in the root system of type $A_2$ is not geometric, since the corresponding chain of ideals $\I=(\{\alpha_1,\alpha_2\},\{\alpha_1,\alpha_2\})$ has $\alpha_1,\alpha_2\in I_1$, but $\alpha_1+\alpha_2\notin I_{1+1}=I_2$.
\end{mdframed}
If $R$ is a dominant $\m$-Shi region define $\theta(R):=(I_1,I_2,\ldots,I_{\m})$ and $\phi(R):=(J_1,J_2,\ldots,J_{\m})$, where
$$I_i:=\{\alpha\in\Phi^+\mid\langle x,\alpha\rangle<i\text{ for all }x\in R\}\text{ and}$$
$$J_i:=\{\alpha\in\Phi^+\mid\langle x,\alpha\rangle>i\text{ for all }x\in R\}$$
for $i\in\{0,1,\ldots,\m\}$.
It is not difficult to verify that $\theta(R)$ is a geometric chain of ideals and that $\phi(R)$ is the corresponding geometric chain of order filters.\\
\\
In fact $\theta$ is a bijection from dominant $\m$-Shi regions to geometric chains of ideals. Equivalently $\phi$ is a bijection from dominant $\m$-Shi regions to geometric chains of order filters \cite[Theorem 3.6]{athanasiadis05refinement}.

\section{Minimal alcoves of $\m$-Shi regions}\label{minsec}
Any alcove of the affine Coxeter arrangement is contained in a unique $\m$-Shi region.
We will soon see that for any $\m$-Shi region $R$ there is a unique alcove $\waf_R\ac\subseteq R$
such that for all $\waf\ac\subseteq R$ and all $\alpha\in\Phi^+$ we have 
$$|k(\waf_R,\alpha)|\leq |k(\waf,\alpha)|.$$
We call $\waf_R\ac$ the \defn{minimal alcove} of $R$. We say that an alcove $\waf\ac$ is an \defn{$\m$-Shi alcove} if it is the minimal alcove of the $\m$-Shi region containing it.
We define $\alc=\{\waf_R\ac:R\text{ is an $\m$-Shi region}\}$ to be the set of $\m$-Shi alcoves.
\begin{figure}[h]
\begin{center}
 \resizebox*{9cm}{!}{\includegraphics{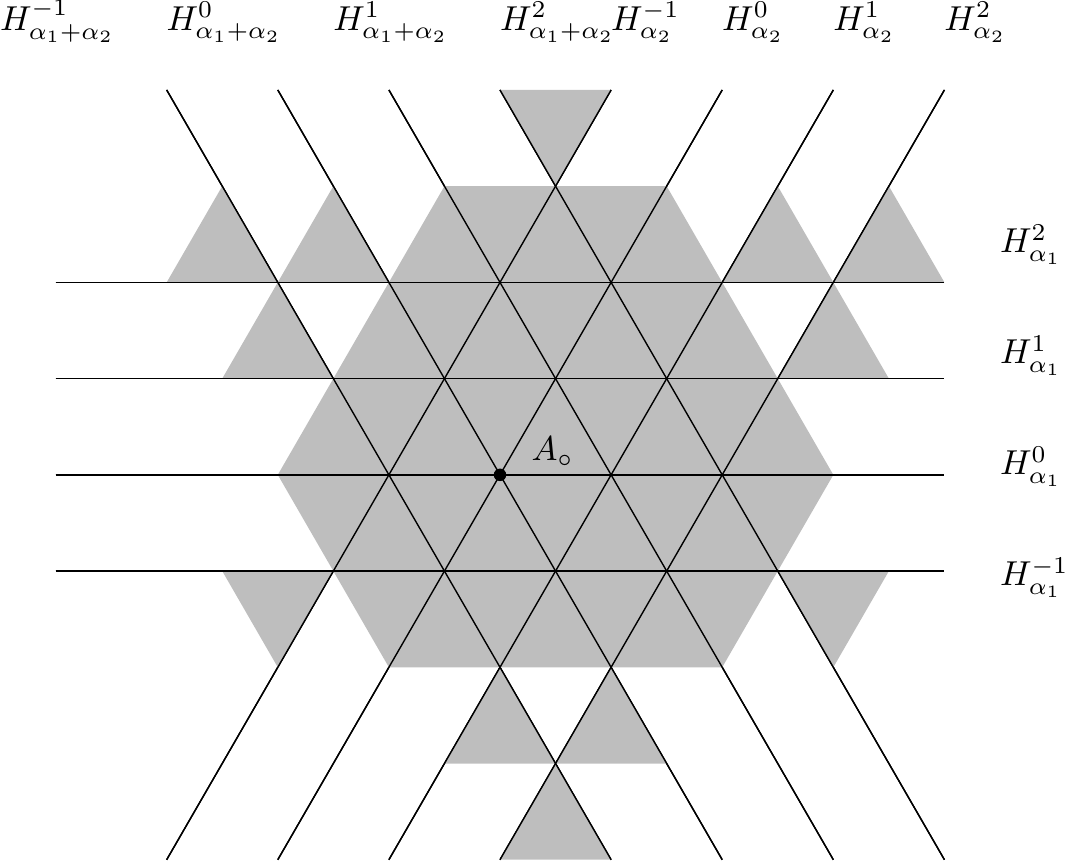}}\\
\end{center}
\caption{The $49$ minimal alcoves of the $2$-Shi arrangement of type $A_2$.}
\end{figure}
\subsection{The address of a dominant $\m$-Shi alcove}\label{add}
We first concentrate on dominant $\m$-Shi regions and their minimal alcoves.
The following lemma from \cite[Theorem 5.2]{shi87alcoves} gives necessary and sufficient conditions for a tuple $(k_{\alpha})_{\alpha\in\Phi^+}$ to be the address of some alcove $\waf\ac$.
\begin{lemma}[\protect{\cite[Lemma 2.3]{athanasiadis05refinement}}]\label{address}
 Suppose that for each $\alpha\in\Phi^+$ we are given some integer $k_{\alpha}$.
 Then there exists $\waf\in\wa$ with $k(\waf,\alpha)=k_{\alpha}$ for all $\alpha\in\Phi^+$ if and only if
 $$k_{\alpha}+k_{\beta}\leq k_{\alpha+\beta}\leq k_{\alpha}+k_{\beta}+1$$
 for all $\alpha,\beta\in\Phi^+$ with $\alpha+\beta\in\Phi^+$.
\end{lemma}
\noindent
For a geometric chain of order filters $\J=(J_1,J_2,\ldots,J_{\m})$ and $\alpha\in\Phi^+$, define
$$k_{\alpha}(\J)=\text{max}\{k_1+k_2+\ldots+k_l:\alpha=\alpha_1+\alpha_2+\ldots+\alpha_l\text{ and }\alpha_i\in J_{k_i}\text{ for all }i\in[l]\}$$
where $k_i\in\{0,1,\ldots,\m\}$ for all $i\in[l]$.\\
\\
It turns out that the integer tuple $(k_{\alpha}(\J))_{\alpha\in\Phi^+}$ satisfies the conditions of Lemma~\ref{address} \cite[Corollary 3.4]{athanasiadis05refinement}, 
so there is a unique $\waf\in\wa$ with 
$$k(\waf,\alpha)=k_{\alpha}(\J)\text{ for all }\alpha\in\Phi^+.$$
The alcove $\waf\ac$ is exactly the minimal alcove $\waf_R\ac$ of the dominant $\m$-Shi region $R:=\phi^{-1}(\J)$ corresponding to $\J$ \cite[Proposition 3.7]{athanasiadis05refinement}.
\subsection{Floors of dominant $\m$-Shi regions and alcoves}\label{fldomshialc}
The floors of a dominant $\m$-Shi region $R$ can be seen in the corresponding geometric chain of order filters $\J:=\phi(R)$ as follows.
If $k$ is a positive integer, a root $\alpha\in\Phi^+$ is called a \defn{rank $k$ indecomposable element} of a geometric chain of order filters $\J=(J_1,J_2,\ldots,J_{\m})$ if the following hold:
\begin{enumerate}
 \item $k_{\alpha}(\J)=k$,
 \item $\alpha\notin J_i+J_j\text{ for any }i,j\in\{0,1,\ldots,m\}\text{ with }i+j=k$ and
 \item if $\alpha+\beta\in J_t$ and $k_{\alpha+\beta}(J)=t\leq m$ for some $\beta\in\Phi^+$, then $\beta\in J_{t-k}$.
\end{enumerate}
The following theorem relates the indecomposable elements of $\J$ to the floors of $R$ and $\waf_R\ac$.
\begin{theorem}[\protect{\cite[Theorem 3.11]{athanasiadis05refinement}}]\label{indfloor}
 If $R$ is a dominant $\m$-Shi region, $\J=\phi(R)$ is the corresponding geometric chain of order filters and $\alpha\in\Phi^+$, then the following are equivalent:
 \begin{enumerate}
 \item $\alpha$ is a rank $k$ indecomposable element of $\J$,
 \item $H_{\alpha}^k$ is a floor of $R$, and
 \item $H_{\alpha}^k$ is a floor of $\waf_R\ac$.
\end{enumerate}
\end{theorem}
\subsection{$\m$-Shi regions and alcoves in other chambers}
The following easy lemma, generalising \cite[Lemma 10.2]{armstrong12parking}, describes what the $\m$-Shi arrangement looks like in each chamber.
\begin{lemma}\label{cham}
 For $w\in W$, the $\m$-Shi hyperplanes that intersect the chamber $wC$ are exactly those of the form $H_{w(\alpha)}^k$ where $\alpha\in\Phi^+$ and either $1\leq k<m$ or $k=m$ and $w(\alpha)\in\Phi^+$.
 \begin{proof}
  If an $\m$-Shi hyperplane $H_{\beta}^k$ with $\beta\in\Phi$ and $1\leq k\leq m$ intersects $wC$, then there is some $x\in wC$ with $\langle x,\beta\rangle=k$.
  So $w^{-1}(x)\in C$ and $\langle w^{-1}(x),w^{-1}(\beta)\rangle=k>0$, thus $\alpha:=w^{-1}(\beta)\in\Phi^+$.
  If $k=m$, then $\beta=w(\alpha)\in\Phi^+$ since otherwise $H_{\beta}^k$ is not an $\m$-Shi hyperplane.\\
  \\
  Conversely, if $\alpha\in\Phi^+$ and either $1\leq k<m$ or $k=m$ and $w(\alpha)\in\Phi^+$, then $H_{w(\alpha)}^k$ is an $\m$-Shi hyperplane.
  Take $x\in C$ with $\langle x,\alpha\rangle=k$. Then $w(x)\in wC$ and $\langle w(x),w(\alpha)\rangle=k$, so $H_{w(\alpha)}^k$ intersects $wC$.
 \end{proof}

\end{lemma}
\noindent
We are now ready for our first main theorem about minimal alcoves of $\m$-Shi regions, which we will use frequently and without mention. It is already known for dominant regions \cite[Proposition 3.7, Theorem 3.11]{athanasiadis05refinement}.
\begin{theorem}\label{min}
 Every region $R$ of the $\m$-Shi arrangement contains a unique minimal alcove $\waf_R\ac$.
 That is, for any $\alpha\in\Phi^+$ and $\waf\in\wa$ such that $\waf A_{\circ}\subseteq R$, we have $|k(\waf_R,\alpha)|\leq |k(\waf,\alpha)|$.
 The floors of $\waf_RA_{\circ}$ are exactly the floors of $R$. 
\begin{proof}
 The concept of the proof is as follows. Start with an $m$-Shi region $R$ contained in the chamber $wC$. Consider $\rd:=w^{-1}R\subseteq C$. This is not in general an $m$-Shi region, but it contains a unique $m$-Shi region $\rmin$ that is closest to the origin. We take its minimal alcove $\wm\ac$ and find that $w\wm\ac$ is the minimal alcove of $R$.\\
 \\
 Suppose $R$ is an $m$-Shi region contained in the chamber $wC$.
 Let $\rd:=w^{-1}R\subseteq C$. Notice that $\rd$ need not itself be an $\m$-Shi region. By Lemma \ref{cham}, the walls of $R$ are of the form $H_{w(\alpha)}^k$ where $\alpha\in\Phi^+$ and either $0\leq k<m$ or $k=m$ and $w(\alpha)\in\Phi^+$. Thus the walls of $\rd=w^{-1}R$ are of the form $H_{\alpha}^k$ with $\alpha\in\Phi^+$ and either $0\leq k< m$ or $k=m$ and $w(\alpha)\in\Phi^+$.
 In particular, they are $m$-Shi hyperplanes. The only $m$-Shi hyperplanes $H$ that may intersect $\rd$ are those such that $w(H)$ is not an $\m$-Shi hyperplane, that is those of the form $H_{\alpha}^m$ with $w(\alpha)\in-\Phi^+$.\\
 \\
 Now suppose $\re$ is a dominant $\m$-Shi region and $\J'=\phi(R')$ is the corresponding geometric chain of order filters.
 Then $\re\subseteq\rd$ if and only if for every $\m$-Shi hyperplane $H_{\alpha}^k$, whenever all of $\rd$ is on one side of $H_{\alpha}^k$, then all of $\re$ is on the same side of it.
 Equivalently, $\re\subseteq\rd$ 
 precisely when for all $1\leq k\leq m$ and $\alpha\in\Phi^+$ we have $\alpha\in J_k'$ if $\langle x,\alpha\rangle>k$ for all $x\in \rd$, and $\alpha\in I_k'$ if $\langle x,\alpha\rangle<k$ for all $x\in \rd$.\\
 \\
 Let $\J=(J_1,J_2,\ldots,J_{\m})$ be the chain of order filters with $\alpha\in J_k$ if and only if $\langle x,\alpha\rangle>k$ for all $x\in \rd$.
 To see that $\J$ is geometric, first note that if $\alpha\in J_i$, $\beta\in J_j$ and $\alpha+\beta\in\Phi^+$, then $\langle x,\alpha+\beta\rangle=\langle x,\alpha\rangle+\langle x,\beta\rangle>i+j$ for all $x\in \rd$, so $\alpha+\beta\in J_{i+j}$.
 Let $\re$ be some $\m$-Shi region contained in $\rd$ and let $\J'=\phi(\re)$ be the corresponding geometric chain of order filters.
 Then $\re$ and $\rd$ are on the same side of every $\m$-Shi hyperplane that does not intersect $\rd$, so in particular $J_k=J_k'$ for $1\leq k<\m$.
 Whenever $\alpha\in J_{\m}$, then $\langle x,\alpha\rangle>m$ for all $x\in\rd$, so $\alpha\in J_{\m}'$. Thus $J_{\m}\subseteq J_{\m}'$.
 If $i+j\leq \m$, assume without loss of generality that $i,j>0$, so that $i,j<\m$ and
 $$(I_i+I_j)\cap\Phi^+=(I'_i+I'_j)\cap\Phi^+\subseteq I'_{i+j}\subseteq I_{i+j},$$
 since $\J'$ is geometric. This shows that $\J$ is geometric.
 Thus there is a dominant region $\rmin=\phi^{-1}(\J)$.
 We clearly have $\alpha\in J_k$ if $\langle x,\alpha\rangle>k$ for all $x\in \rd$, and whenever $\langle x,\alpha\rangle<k$ for all $x\in \rd$, then $\alpha\in I'_k\subseteq I_k$.
 Thus $\rmin\subseteq\rd$.
 Observe that $k_{\alpha}(\J)\leq k_{\alpha}(\J')$ for all $\alpha\in\Phi^+$. Also note that $\langle x,\alpha\rangle>k_{\alpha}(\J)$ for all $x\in \rd$.\\
 \\
 Let $\waf_{\rm{min}}A_{\circ}$ be the minimal alcove of $\rmin$ \cite{athanasiadis05refinement}. Thus we have $k(\waf_{\rm{min}},\alpha)=k_{\alpha}(\J)$ for all $\alpha\in\Phi^+$. So if $\waf A_{\circ}$ is any alcove contained in $\rd$, say $\waf A_{\circ}\subseteq\re$ for some $\m$-Shi region $\re\subseteq\rd$, then if $\J'=\phi(\re)$ we have $k(\waf_{\rm{min}},\alpha)=k_{\alpha}(\J)\leq k_{\alpha}(\J')\leq k(\waf,\alpha)$ for all $\alpha\in\Phi^+$.\\
 \\
 So if we define $\waf_R:=w\wm$, then $\waf_RA_{\circ}\subseteq R$ and $k(\waf_R,\alpha)=k(\wm,w^{-1}(\alpha))$ for all $\alpha\in\Phi$.
 If $\waf A_{\circ}$ is any alcove contained in $R$, $\alpha\in\Phi^+$ and $w^{-1}(\alpha)\in\Phi^+$, then $$k(\waf,\alpha)=k(w^{-1}\waf,w^{-1}(\alpha))\geq k(\wm,w^{-1}(\alpha))=k(\waf_R,\alpha)\text{,}$$ since $w^{-1}\waf A_{\circ}\subseteq \rd$. Note that in this case $k(\waf_R,\alpha)=k(\wm,w^{-1}(\alpha))\geq0$, since $w^{-1}(\alpha)\in\Phi^+$ and $\wm\ac$ is dominant.
 If instead $w^{-1}(\alpha)\in-\Phi^+$, then 
 \begin{multline*}
 k(\waf,\alpha)=k(w^{-1}\waf,w^{-1}(\alpha))=-k(w^{-1}\waf,-w^{-1}(\alpha))-1\\
 \leq -k(\wm,-w^{-1}(\alpha))-1=k(\wm,w^{-1}(\alpha))=k(\waf_R,\alpha)\text{.} 
 \end{multline*} 
 Note that in this case, $k(\waf_R,\alpha)=-k(\wm,-w^{-1}(\alpha))-1<0$. So either way we have $|k(\waf_R,\alpha)|\leq |k(\waf,\alpha)|$.\\
 \\
 Suppose $H_{\alpha}^k$ is a floor of $\waf_R\ac$. Then it is the only hyperplane separating $s_{\alpha}^k\waf_RA_{\circ}$ from $\waf_RA_{\circ}$. Thus $k(s_{\alpha}^k\waf_R,\beta)=k(\waf_R,\beta)$ for all $\beta\neq\pm\alpha$ and $|k(s_{\alpha}^k\waf_R,\alpha)|=|k(\waf_R,\alpha)|-1$. Since $\waf_R\ac$ is the minimal alcove of $R$ this implies that $s_{\alpha}^k\waf_RA_{\circ}$ is not contained in $R$. Thus $H_{\alpha}^k$ must be an $m$-Shi hyperplane, and therefore a floor of $R$.\\
 \\
 Suppose $H_{\alpha}^k$ is a floor of $\rd$, where $\alpha\in\Phi^+$. Then we claim that $\alpha$ is a rank $k$ indecomposable element of $\J=\phi(\rmin)$.
 To see this, first note that $\langle x,\alpha\rangle>k$ for all $x\in\rd$, so $\alpha\in J_k$. Also, $\langle x,\alpha\rangle<k+1$ for some $x\in\rd$, so $k_{\alpha}(\J)=k$.
 Suppose $\alpha=\beta+\gamma$ with $\beta\in J_i$ and $\gamma\in J_j$ and $i+j=k$. Then $\langle x,\beta\rangle>i$ and $\langle x,\gamma\rangle>j$ imply that $\langle x,\alpha\rangle>k$ for $x\in \rd$, so $H_{\alpha}^k$ does not support a facet of $R$, a contradiction.
 If $\alpha+\beta\in J_t$ and $k_{\alpha+\beta}(\J)=t\leq m$ for some $\beta\in\Phi^+$ then we have $\langle x,\alpha+\beta\rangle>t$ for all $x\in\rd$ so we cannot have $\langle x,\beta\rangle<t-k$ for all $x\in\rd$, since together they would imply that $\langle x,\alpha\rangle>k$ for all $x\in\rd$, so $H_{\alpha}^k$ would not support a facet of $R$.
 Since $t-k<m$, the hyperplane $H_{\beta}^{t-k}$ does not intersect $\rd$, so this implies that $\langle x,\beta\rangle>t-k$ for all $x\in\rd$, so $\beta\in J_{t-k}$. This verifies the claim.
 From the fact that $\alpha$ is a rank $k$ indecomposable element of $\J$ it follows that $H_{\alpha}^k$ is a floor of $\wm A_{\circ}$ by Theorem \ref{indfloor}.\\
 \\
 Now suppose that $H_{\alpha}^k$ is a floor of $R$. Then $H_{w^{-1}(\alpha)}^k$ is a floor of $\rd$ and thus a floor of $\wm A_{\circ}$.
 So $H_{\alpha}^k$ is a floor of $\waf_RA_{\circ}=w\wm A_{\circ}$.
\end{proof}

\end{theorem}
\subsection{$\m$-Shi alcoves and $(mh+1)$-stable affine Weyl group elements}
The following lemma characterises the $\m$-Shi alcoves. It is a straightforward generalisation of \cite[Proposition 7.3]{shi87sign}.
\begin{lemma}\label{minflo}
 An alcove $\waf A_{\circ}$ is an $\m$-Shi alcove if and only if all its floors are $m$-Shi hyperplanes.
 \begin{proof}
  The forward implication is immediate from Theorem \ref{min}.\\
  \\
  For the backward implication, we prove the contrapositive: we show that every alcove that is not an $\m$-Shi alcove has a floor that is not an $\m$-Shi hyperplane.
  So suppose $\waf A_{\circ}$ is an alcove contained in an $m$-Shi region $R$, and $\waf\neq \waf_R$.
  Consider the set 
  \begin{multline*}
  K=\{x\in V\mid k(\waf_R,\alpha)<\langle x,\alpha\rangle< k(\waf,\alpha)+1\text{ for all }\alpha\in\Phi\text{ with }k(\waf_R,\alpha)\geq0\\
  \text{ and }k(\waf,\alpha)<\langle x,\alpha\rangle< k(\waf_R,\alpha)+1\text{ for all }\alpha\in\Phi\text{ with }k(\waf_R,\alpha)<0\}.
  \end{multline*}
  Then any alcove $\waf'A_{\circ}$ has either $\waf'A_{\circ}\subseteq K$ or $\waf'A_{\circ}\cap K=\varnothing$. 
  For $\alpha\in\Phi$, we have 
  \[k(\waf_R,\alpha)\leq k(\waf',\alpha)\leq k(\waf,\alpha)\]
  whenever $\waf'A_{\circ}\subseteq K$ and $k(\waf_R,\alpha)\geq0$.
  Similarly 
  \[k(\waf,\alpha)\leq k(\waf',\alpha)\leq k(\waf_R,\alpha)\]
  whenever $\waf'A_{\circ}\subseteq K$ and $k(\waf_R,\alpha)<0$.
  Thus any hyperplane of the affine Coxeter arrangement that separates two alcoves contained in $K$ also separates $\waf_R\ac$ and $\waf\ac$. Since no $\m$-Shi hyperplane separates $\waf_R\ac$ and $\waf\ac$, no $\m$-Shi hyperplane separates two alcoves contained in $K$.
  Since $K$ is convex, there exists a sequence $(\waf_1,\waf_2,\ldots,\waf_l)$ with $\waf_1=\waf$, $\waf_l=\waf_R$, and $\waf_iA_{\circ}\subseteq K$ for all $i\in[l]$, such that $\waf_iA_{\circ}$ shares a facet with $\waf_{i+1}A_{\circ}$ for all $i\in[l-1]$.
  So the supporting hyperplane of the common facet of $\waf_1A_{\circ}=\waf A_{\circ}$ and $\waf_2A_{\circ}$ is a floor of $\waf A_{\circ}$ which is not an $\m$-Shi hyperplane.
 \end{proof}

\end{lemma}
\noindent
We can now relate the $\m$-Shi alcoves to the Fuß-Catalan ($\rat=mh+1$) case of the set of $\rat$-stable affine Weyl group elements $\wa^{\rat}$ defined in Section \ref{stable}.
\begin{theorem}\label{shistable}
 An alcove $\waf A_{\circ}$ is an $\m$-Shi alcove if and only if $\waf\in\wa^{mh+1}$.
 \begin{proof}
  First note that
  \[\ar_{mh+1}=\ar_1+m\delta=\widetilde{\simp}+m\delta.\]
  Suppose $\waf\ac$ is an $\m$-Shi alcove and take 
  \[\alpha+k\delta\in\waf(\ar_{mh+1})=\waf(\widetilde{\simp}+m\delta)=\waf(\widetilde{\simp})+m\delta.\]
  So $\alpha+(k-m)\delta\in\waf(\widetilde{\simp})$ and thus $\waf^{-1}(-\alpha+(m-k)\delta)\in-\widetilde{\simp}$.
  If $k\geq m$ then $\alpha+k\delta\in\ar^+$. Otherwise by Lemma \ref{arfl} $H_{-\alpha}^{k-m}$ is a floor of $\waf\ac$ and thus by Lemma \ref{minflo} an $\m$-Shi hyperplane.
  So $k\geq0$ and $k>0$ if $\alpha\in-\Phi^+$. Thus $\alpha+k\delta\in\ar^+$. So $\waf(\ar_{mh+1})\subseteq\ar^+$ and therefore $\waf\in\wa^{mh+1}$.\\
  \\
  Conversely suppose $\waf\in\wa^{mh+1}$ and $H_{\alpha}^{-k}$ is a floor of $\waf\ac$ where $k>0$. Then by Lemma \ref{arfl} we have $\waf^{-1}(\alpha+k\delta)\in-\widetilde{\simp}$.
  Thus 
  \[-\alpha+(m-k)\delta\in\waf(\ar_{mh+1})\subseteq\ar^+.\]
  So $k\leq m$ and $k<m$ if $\alpha\in\Phi^+$. Thus $H_{\alpha}^{-k}$ is an $\m$-Shi hyperplane. So all floors of $\waf\ac$ are $\m$-Shi hyperplanes and thus by Lemma \ref{minflo} $\waf\ac$ is an $\m$-Shi alcove.
 \end{proof}

\end{theorem}
\noindent
From Theorem \ref{shistable} and Corollary \ref{ratr} we deduce the following theorem.
\begin{theorem}
 The number of $\m$-Shi alcoves equals $(mh+1)^r$.
\end{theorem}
\noindent
We use Theorem \ref{min} to recover the following result, originally proved by Yoshinaga using the theory of free arrangements \protect{\cite[Theorem 1.2]{yoshinaga04characterization}}.
\begin{theorem}
 The number of $\m$-Shi regions equals $(mh+1)^r$.
\end{theorem}
\section{Diagonally labelled Dyck paths}\label{diag}
From now on, take $\Phi$ be the root system of type $A_{n-1}$ and set $m=1$. Thus we write $\mathsf{Alc}_{\Phi}$ for $\mathsf{Alc}^1_{\Phi}$, $\mathsf{Park}(\Phi)$ for $\mathsf{Park}^{(1)}(\Phi)$, and so on.
Our aim is to relate our zeta map $\zeta$ from Theorem \ref{zeta} to the zeta map $\zeta_{HL}$ of Haglund and Loehr \cite[Theorem 5.6]{haglund08catalan}.
To do this, we give a combinatorial model for $\mathsf{Park}(\Phi)$ in terms of \defn{diagonally labelled Dyck paths}.
A \defn{Dyck path} of length $n$ is a lattice path in $\Z^2$ consisting of North and East steps that goes from $(0,0)$ to $(n,n)$ and never goes below the diagonal $x=y$.
Equivalently, it is a rational $(n+1,n)$-Dyck path with its final step (always an East step) removed. Thus we will not distinguish between Dyck paths and $(n+1,n)$-Dyck paths.\\
\begin{figure}[h]
\begin{center}
\begin{tikzpicture}[scale=.6]
\begin{scope}
	\draw[gray] (0,0) grid (6,6);
	\draw[very thick] (0,0)--(0,1)--(1,1)--(1,4)--(3,4)--(3,6)--(6,6);
	\draw[xshift=5mm,yshift=5mm]
		(0,0) node{\large{$3$}}
		(1,1) node{\large{$6$}}
		(2,2) node{\large{$1$}}
		(3,3) node{\large{$2$}}
		(4,4) node{\large{$4$}}
		(5,5) node{\large{$5$}}
		(0,1) node{\large{$\bullet$}}
		(2,4) node{\large{$\bullet$}};
\end{scope}
\end{tikzpicture}
\caption{A diagonally labelled Dyck path $(w,D)$ of length $6$. 
The valleys are marked by dots.}
\end{center}
\end{figure}
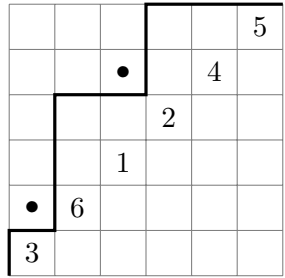
\\
A pair $(i,j)$ of positive integers is called a \defn{valley} of a Dyck path $D$ if the $i$-th East step of $D$ is immediately followed by its $j$-th North step.
The pair $(w,D)$ with $w\in S_n$ is called a diagonally labelled Dyck path if $w(i)<w(j)$ whenever $(i,j)$ is a valley of $D$.
We say that the valley $(i,j)$ is labelled $(w(i),w(j))$.
We think of $w$ as labeling the boxes crossed by the diagonal between $(0,0)$ and $(n,n)$ from bottom to top.
So the condition on $w$ is that whenever an East step of $D$ is followed by a North step, the label below the East step is less than the label to the right of the North step.
We write $\D_n$ for the set of diagonally labelled Dyck paths on length $n$.

\section{The combinatorial zeta map}\label{zetahl}
The zeta map $\zeta_{HL}$ of Haglund and Loehr is a bijection from $\pf_n$ to $\D_n$.
It is defined as follows. The first ingredient is the zeta map $\zeta_H$ of Haglund \cite[Theorem 3.15]{haglund08catalan}. It is a bijection from the set of Dyck paths of length $n$ to itself.
Given a Dyck path $P$, iterate the following procedure for $i=0,1,\ldots,n$: go through the area vector of $P$ from left to right and draw a North step for every $i$ you see and an East step for every $i-1$.
This gives a Dyck path $\zeta_H(P)$.\\
\\
The second ingredient is the \defn{diagonal reading word} $\drw(P,\sigma)$ of the vertically labelled Dyck path $(P,\sigma)$.
It is given by first reading the labels of rows of area $0$ from bottom to top, then the labels of rows of area $1$ from bottom to top, and so on.\\
\begin{figure}[h]
\begin{center}
\begin{tikzpicture}[scale=.6]
\begin{scope}
	\draw[gray] (0,0) grid (5,5);
	\draw[very thick] (0,0)--(0,3)--(2,3)--(2,4)--(3,4)--(3,5)--(5,5);
	\draw (2.5,-.2) node[anchor=north]{$a=(0,1,2,1,1)$};
	\draw[->,thick] (6,2.5)--(7,2.5);
	\draw[xshift=-5mm,yshift=5mm]
		(0,0) node{\large{$1$}}
		(0,1) node{\large{$2$}}
		(0,2) node{\large{$4$}}
		(2,3) node{\large{$3$}}
		(3,4) node{\large{$5$}};
\end{scope}
%
%
%
\begin{scope}[xshift=8cm]
	\draw[gray] (0,0) grid (5,5);
	\draw[very thick] (0,0)--(0,1)--(1,1)--(1,4)--(2,4)--(2,5)--(5,5);
	\draw[->,thick] (-2,2.5)--(-1,2.5);
	\draw[xshift=5mm,yshift=5mm]
		(0,0) node{\large{$1$}}
		(1,1) node{\large{$2$}}
		(2,2) node{\large{$3$}}
		(3,3) node{\large{$5$}}
		(4,4) node{\large{$4$}}
		(0,1) node{\large{$\bullet$}}
		(1,4) node{\large{$\bullet$}};
\end{scope}
\end{tikzpicture}
\caption{The zeta map $\zeta_{HL}$ maps a vertically labelled Dyck path $(P,\sigma)$ to a diagonally labelled Dyck path $(w,D)$.}
\label{Figure:zetaA}
\end{center}
\end{figure}
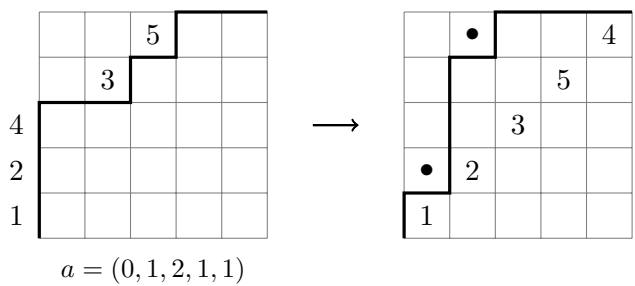
\\
We define $\zeta_{HL}(P,\sigma):=(\drw(P,\sigma),\zeta_H(P))$. The following is an important property of $\zeta_{HL}$.
It inspired the generalisation $\zeta$ of $\zeta_{HL}$ that will be defined later.
\begin{theorem}[\protect{\cite[Section 5.2]{armstrong14rational}}]\label{armzeta}
 For any vertically labelled Dyck path $(P,\sigma)$ and any pair of positive integers $(b,c)$, the diagonally labelled Dyck path $\zeta_{HL}(P,\sigma)=(\drw(P,\sigma),\zeta_H(P))$ has a valley labelled $(b,c)$ if and only if $(P,\sigma)$ has a rise labelled $(b,c)$.
 \begin{proof}
  Let $i\in[n]$ be an index. Let $a=a_i$ be the area of the $i$-th row of $(P,\sigma)$ and let $\sigma(i)$ be its label.
  Suppose that $\sigma(i)=\drw(P,\sigma)(j)$ is the $j$-th label being read in the diagonal reading word of $(P,\sigma)$.
  That means that there are exactly $j-1$ rows that have either smaller area than row $i$, or have the same area $a$ and are nearer the bottom.
  Thus in the construction of $\zeta_H(P)$, the $j$-th North step is drawn when the area vector entry $a_i$ is read in iteration $a$, and the $j$-th East step is drawn when the area vector entry $a_i$ is read in iteration $a+1$.\\
  \\
  Suppose that $i$ is a rise of $(P,\sigma)$, labelled $(\sigma(i),\sigma(i+1))$. Then $a_{i+1}=a+1$. So in particular, the label $\sigma(i+1)$ is read later in the diagonal reading word than the label $\sigma(i)$.
  That is, if $\sigma(i)=\drw(P,\sigma)(j)$ and $\sigma(i+1)=\drw(P,\sigma)(k)$, then $j<k$. In the construction of $\zeta_H(P)$, in the $(a+1)$-st iteration the $j$-th East step is drawn when $a_i$ is read, and immediately afterwards the $k$-th North step is drawn when $a_{i+1}$ is read.
  Thus $(j,k)$ is a valley of $\zeta_{H}(P)$. In $\zeta_{HL}(P,\sigma)=(\drw(P,\sigma),\zeta_H(P))$ it is labelled $(\drw(P,\sigma)(j),\drw(P,\sigma)(k))=(\sigma(i),\sigma(i+1))$.\\
  \\
  Conversely suppose that $(j,k)$ is a valley of $\zeta_{HL}(P,\sigma)$ that is labelled by $(b,c)=(\drw(P,\sigma)(j),\drw(P,\sigma)(k))$.
  That is, the $j$-th East step is immediately followed by the $k$-th North step. Since every iteration except for the $0$-th starts with an East step, both steps must have been drawn in the same iteration, say iteration $a+1$.
  Then there is some row $i$ with area $a$ and a row $j>i$ with area $a+1$ such that no row between $i$ and $j$ has area either $a$ or $a+1$.
  This implies $j=i+1$, so $i$ is a rise of $(P,\sigma)$. As above, it follows that $\sigma(i)=\drw(P,\sigma)(j)$ and $\sigma(i+1)=\drw(P,\sigma)(k)$, so the rise $i$ is labelled $(\sigma(i),\sigma(i+1))=(b,c)$.
 \end{proof}

\end{theorem}
\noindent
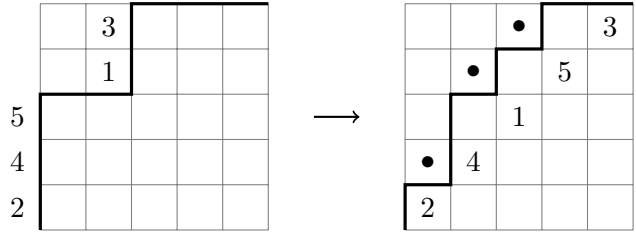
\begin{figure}[h]
\begin{center}
\begin{tikzpicture}[scale=.6]
\begin{scope}
	\draw[gray] (0,0) grid (5,5);
	\draw[very thick] (0,0)--(0,3)--(2,3)--(2,5)--(5,5);
	\draw[->,thick] (6,2.5)--(7,2.5);
	\draw[xshift=-5mm,yshift=5mm]
		(0,0) node{\large{$2$}}
		(0,1) node{\large{$4$}}
		(0,2) node{\large{$5$}}
		(2,3) node{\large{$1$}}
		(2,4) node{\large{$3$}};
\end{scope}
\begin{scope}[xshift=8cm]
	\draw[gray] (0,0) grid (5,5);
	\draw[very thick] (0,0)--(0,1)--(1,1)--(1,3)--(2,3)--(2,4)--(3,4)--(3,5)--(5,5);
	\draw[xshift=5mm,yshift=5mm]
		(0,0) node{\large{$2$}}
		(1,1) node{\large{$4$}}
		(2,2) node{\large{$1$}}
		(3,3) node{\large{$5$}}
		(4,4) node{\large{$3$}}
		(0,1) node{\large{$\bullet$}}
		(1,3) node{\large{$\bullet$}}
		(2,4) node{\large{$\bullet$}};
\end{scope}
\end{tikzpicture}
\caption{The zeta map $\zeta_{HL}$: The vertically labelled Dyck path $(P,\sigma)$ on the left is mapped to the diagonally labelled Dyck path $(w,D)$ on the right.}
\end{center}
\end{figure}
\\
A Dyck path is uniquely determined by its valleys. Thus Theorem \ref{armzeta} gives rise to an alternative description of $\zeta_{HL}$. We have $\zeta_{HL}(P,\sigma)=(\drw(P,\sigma),D)$, where $D$ is the unique Dyck path such that $(\drw(P,\sigma),D)$ has a valley labelled $(b,c)$ if and only if $(P,\sigma)$ has a rise labelled $(b,c)$.\\
\\
\textbf{Remark.} The combinatorial zeta map $\zeta_{H}$ on \defn{unlabelled} Dyck paths has a generalisation to rational $p/n$-Dyck paths \cite[Section 4.2]{armstrong14sweep} that Williams recently proved to be a bijection \cite{williams15sweep}.
We do not know whether there is a uniform generalisation of the zeta map at the rational level of generality, but the next section will provide one for the Fuß-Catalan case $p=mh+1$.
\section{The uniform zeta map}\label{zetasec}
We will describe a uniform generalisation $\zeta$ of the zeta map $\zeta_{HL}$ of Haglund and Loehr to all irreducible crystallographic root systems, and also to the Fuß-Catalan level of generality.
Recall from Section \ref{ptt} that $\pf_{n}=\pf_{n+1/n}$ is naturally in bijection with the finite torus $\Q/(h+1)\Q$ of the root system of type $A_{n-1}$.
The Fuß-Catalan generalisation is given by $\Q/(mh+1)\Q$ for $m$ a positive integer. It remains to find an interpretation of $\D_n$ in terms of the root system of type $A_{n-1}$, and also a Fuß-Catalan generalisation. The next section provides the appropriate uniform generalisation of $\D_n$ to all irreducible crystallographic root systems and also to the Fuß-Catalan level, though a demonstration of this fact will have to wait until Section \ref{zeta+zeta}.
\subsection{The nonnesting parking functions}\label{nnpark}
The set of \defn{nonnesting parking functions} $\park$ of an irreducible crystallographic root system $\Phi$ was introduced by Armstrong, Reiner and Rhoades \cite{armstrong12parking}.
It was defined in order to combine two desirable properties: being naturally in bijection with the Shi regions of $\Phi$ and carrying a natural $W$-action such that $\park$ is isomorphic to the finite torus $\Q/(h+1)\Q$ as a $W$-set.
The latter property justifies the name ``nonnesting parking functions'', since $\park\cong\Q/(n+1)\Q\cong\pf_n$ if $\Phi$ is of type $A_{n-1}$.\\
\\
The set of \defn{$\m$-nonnesting parking functions} $\parkm$ is the natural Fuß-Catalan generalisation of $\park$. It was introduced by Rhoades \cite{rhoades12parking}.
Given a geometric chain $\J$ of $\m$ order filters in the root poset of $\Phi$, define $\ind(\J)$ as the set of rank $\m$ indecomposable elements of $\J$ (see Section \ref{fldomshialc}) and let
\[W_{\J}=\langle\{s_{\alpha}:\alpha\in\ind(\J)\}\rangle.\]
The set $\parkm$ of \defn{$\m$-nonnesting parking functions} of $\Phi$ is the set of equivalence classes of pairs $(w,\J)$ with $w\in W$ and $\J$ a geometric chain of $\m$ order filters under the equivalence relation
$$(w_1,\J_1)\sim(w_2,\J_2)\text{ if and only if }\J_1=\J_2\text{ and }w_1W_{\J_1}=w_2W_{\J_1}.$$
$\parkm$ is endowed with a left action of $W$ defined by
$$u\cdot[w,\J]:=[uw,\J]$$
for $u\in W$.\\
\\
All the rank $\m$ indecomposable elements of a geometric chain of order filters $\J=(J_1,J_2,\ldots,J_{\m})$ are minimal elements of $J_{\m}$ by \cite[Lemma 1]{thiel13hf}.
Thus in particular they are incomparable, that is they form an antichain in the root poset. So there is some $u\in W$ with $I:=u(\ind(\J))\subseteq\simp$ by \cite[Theorem~6.4]{sommers05stable}.
In particular, $W_{\J}=u^{-1}W_Iu$ is a \emph{parabolic subgroup} of $W$ and any left coset $wW_{\J}$ of $W_{\J}$ in $W$ has a unique representative $w'$ such that $w'(\ind(\J))\subseteq\Phi^+$.
\begin{lemma}\label{indinv}
 For any dominant $\m$-Shi region $R$ corresponding to a geometric chain of order filters $\J$ we have
 \[\ind(\J)=\waf_R(\ar_{mh+1})\cap\Phi.\]
 \begin{proof}
  For $\alpha\in\Phi^+$, we have the following chain of equivalences.
  \begin{align*}
   &\alpha\in\ind(\J)\\
   &\Leftrightarrow H_{\alpha}^{\m}\text{ is a floor of }\waf_R\ac\\
   &\Leftrightarrow \waf_R^{-1}(-\alpha+\m\delta)\in-\widetilde{\Delta}\\
   &\Leftrightarrow \waf_R^{-1}(-\alpha)\in-\widetilde{\Delta}-\m\delta=-\ar_{mh+1}\\
   &\Leftrightarrow \alpha\in\waf_R(\ar_{mh+1}).
  \end{align*}
  Here we used Theorem \ref{indfloor} and Lemma \ref{arfl}.
 \end{proof}

\end{lemma}
\noindent
The following natural bijection relates the $\m$-nonnesting parking functions to the minimal alcoves of the $\m$-Shi arrangement, or equivalently the $(mh+1)$-stable affine Weyl group elements.
\begin{theorem}\label{pta}
 The map
 \begin{align*}
   \pta:\parkm&\rightarrow\wa^{mh+1}\\
   [w,\J]&\mapsto w'\waf_R
 \end{align*}
 is a well-defined bijection. Here $w'$ is the unique representative of $wW_{\J}$ with $w'(\ind(\J))\subseteq\Phi^+$ and $R:=\phi^{-1}(\J)$ is the dominant $\m$-Shi region corresponding to $\J$.
 \begin{proof}
 
  The map $\Theta$ is well-defined, since if $[w_1,\J_1]=[w_2,\J_2]$ then $\J_1=\J_2$ and $w_1W_{\J_1}=w_2W_{\J_1}$, so $w_1'=w_2'$.
  Therefore $w_1'\waf_{R_1}=w_2'\waf_{R_2}$.\\
  \\
  To see that $w'\waf_R\in\wa^{mh+1}$, note that $w'(\waf_R(\ar_{mh+1})\cap\Phi)=w'(\ind(\J))\subseteq\Phi^+$ using Lemma \ref{indinv}.
  Thus $w'\waf_R(\ar_{mh+1})\subseteq\ar^+$ and therefore $w'\waf_R\in\wa^{mh+1}$.\\
  \\
  To see that $\pta$ is injective, suppose that $\pta([w_1,\J_1])=w_1'\waf_{R_1}=w_2'\waf_{R_2}=\pta([w_2,\J_2])$.
  Now $w_1'\waf_{R_1}\ac\subseteq w_1'C$ and $w_2'\waf_{R_2}\ac\subseteq w_2'C$, so $w_1'=w_2'$. Thus $\waf_{R_1}\ac=\waf_{R_2}\ac$ and therefore $\J_1=\J_2$.
  We also get that $w_1W_{R_1}=w_1'W_{R_1}=w_2'W_{R_1}=w_2W_{R_1}$, so $[w_1,\J_1]=[w_2,\J_2]$.\\
  \\
  To see that $\pta$ is surjective, note that if $\waf\in\wa^{mh+1}$, say with $\waf\ac\subseteq wC$, then by Lemma \ref{ratdom} we have $w^{-1}\waf_R\in\wa^{mh+1}_{\mathrm{dom}}$.
  Thus by Lemma \ref{shistable} $w^{-1}\waf\ac$ is the minimal alcove of a dominant $\m$-Shi region $\rd$ corresponding to some geometric chain of order filters $\J$.
  Furthermore 
  \[w(\ind(\J))=w(\waf_{\rd}(\ar_{mh+1})\cap\Phi)=\waf(\ar_{mh+1})\cap\Phi\subseteq\Phi^+\]
  using Lemma \ref{indinv} and that $\waf\in\wa^{mh+1}$.
  Thus $\pta([w,\rd])=w\waf_{\rd}=\waf$.
 \end{proof}

\end{theorem}
\noindent
A similar bijection using ceilings instead of floors was given for the special case where $\m=1$ in \cite[Proposition 10.3]{armstrong12parking}.
Note that the proof furnishes a description of $\pta^{-1}$: we have $\pta^{-1}(\waf)=[w,\rd]$ where $\waf\ac\in wC$ and $\rd$ is the $m$-Shi region containing $w^{-1}\waf\ac$.
\subsection{$\m$-nonnesting parking functions and the finite torus}
In \cite[Proposition 9.9]{rhoades12parking} it is shown that there is a $W$-set isomorphism\footnote{Rhoades mistakenly writes the root lattice $Q$ in place of the coroot lattice $\Q$. However, his result still stands as written: it turns out that $Q/(mh+1)Q$ and $\Q/(mh+1)\Q$ are isomorphic as $W$-sets.}, that is a bijection that commutes with the action of $W$, from $\parkm$ to $\Q/(mh+1)\Q$.
The following theorem makes this isomorphism explicit.
\begin{theorem}\label{pttt}
 The map 
 \begin{align*}
  \ptt:\parkm&\rightarrow\Q/(mh+1)\Q\\
  [w,\J]&\mapsto w\waf_R\waf_{mh+1}^{-1}\cdot0+(mh+1)\Q,
 \end{align*}
where $R$ is the dominant $\m$-Shi region corresponding to $\J$, is a $W$-set isomorphism. In addition, we have $\ptt=\A\circ\pta$.
\begin{proof}
 We will first show that $\ptt=\A\circ\pta$.
 First note that by Theorem \ref{stab} and Lemma \ref{indinv} we have
 \begin{align*}
  &\mathsf{Stab}(\waf_R\waf_{mh+1}^{-1}\cdot0+(mh+1)\Q)\\
  &=\mathsf{Stab}(\A(\waf))\\
  &=\langle\{s_{\beta}:\beta\in\waf(\ar_{mh+1})\cap\Phi\}\rangle\\
  &=\langle\{s_{\beta}:\beta\in\ind(\J)\}\rangle\\
  &=W_{\J}
 \end{align*}
 Let $w'$ be the unique element of $wW_{\J}$ with $w'(\ind(\J))\subseteq\Phi^+$. We calculate that 
 \begin{align*}
  \ptt([w,\J])&=w\waf_R\waf_{mh+1}^{-1}\cdot0+(mh+1)\Q\\
  &=w'\waf_R\waf_{mh+1}^{-1}\cdot0+(mh+1)\Q\\
  &=\A(w'\waf_R\ac)\\
  &=\A(\pta([w,\J])),
 \end{align*}
 using that
 $w^{-1}w'\in W_{\J}=\mathrm{Stab}(\waf_R\waf_{mh+1}^{-1}\cdot0+(mh+1)\Q)$. 
 So $\ptt=\A\circ\pta$ is a well-defined bijection. Since for $u\in W$ we have
 \begin{align*}
  \ptt(u\cdot[w,\J])&=\ptt([uw,\J])\\
  &=uw\waf_R\waf_{mh+1}^{-1}\cdot0+(mh+1)\Q\\
  &=u\cdot\ptt([w,\J])
 \end{align*}
 we see that $\ptt$ is a $W$-set isomorphism.
 \end{proof}
 \end{theorem}
\noindent
We define the \defn{zeta map} as $\zeta:=\ptt^{-1}=\pta^{-1}\circ\A^{-1}$.
\begin{theorem}\label{zeta}
 The map $\zeta$ is a $W$-set isomorphism from $\Q/(mh+1)\Q$ to $\parkm$.
\end{theorem}
\section{The uniform zeta map and the combinatorial zeta map}\label{zeta+zeta}
Our aim for this section is to relate our zeta map $\zeta$ from Theorem \ref{zeta} to the combinatorial zeta map $\zeta_{HL}$ of Haglund and Loehr introduced in Section \ref{zetahl}.
So we let $\Phi$ be the root system of type $A_{n-1}$ and set $m=1$.
The content of this section may be captured in the following commutative diagram of bijections:
\begin{center}
\begin{tikzpicture}
  \matrix (m) [matrix of math nodes,row sep=3em,column sep=6em,minimum width=2em]
  {
     \park & \widetilde{W}^{n+1}=\widetilde{S}_n^{n+1} & \Q/(n+1)\Q\\
     \D_n & {} & \pf_n\\};
  \path[-stealth]
    (m-1-1) edge node [left] {$\ptd$} (m-2-1)
    (m-1-1.east|-m-1-2) edge node [above] {$\pta$} (m-1-2)
    (m-2-3.west|-m-2-1) edge node [above] {$\zeta_{HL}$} (m-2-1)
    (m-1-2) edge node [left] {$\A_{GMV}$} (m-2-3)
            edge node [above] {$\atd$} (m-2-1)
    (m-1-2.east|-m-1-3) edge node [above] {$\A$} (m-1-3)
    (m-1-3) edge node [right] {$\ttpf$} (m-2-3)
    (m-1-3) edge[bend right] node [above] {$\zeta$} (m-1-1);
\end{tikzpicture} 
\end{center}
The first thing we need to do is introduce the bijections $\ptd$ and $\atd$ that relate $\park$ and $\widetilde{S}_n^{n+1}$ to the set $\D_n$ of diagonally labelled Dyck paths of length $n$.
\subsection{Nonnesting parking functions as diagonally labelled Dyck paths}
It is well-known that one can encode the Shi regions of type $A_{n-1}$ as diagonally labelled Dyck paths \cite[Theorem 3]{armstrong13hyper}.
We take a slightly different approach, and instead view diagonally labelled Dyck paths as encoding nonnesting parking functions.\\
\\
Consider the part of the integer grid $\Z^2$ with $0\leq x,y\leq n$.
We think of the boxes above the diagonal $x=y$ as corresponding to the roots in $\Phi^+$.
Say $[i,j]$ is the box whose top right corner is the lattice point $(i,j)$.
If $i<j$ we view $[i,j]$ as corresponding to the root $e_i-e_j\in\Phi^+$.
So we have $\alpha\leq\beta$ in the root poset if and only if the box corresponding to $\alpha$ is weakly to the right and weakly below the box corresponding to $\beta$.\\
\begin{figure}[h]
\begin{center}
\begin{tikzpicture}[scale=1.1]
\begin{scope}
	\draw[gray] (0,0) grid (6,6);
	\draw[very thick] (0,0)--(0,1)--(1,1)--(1,4)--(3,4)--(3,6)--(6,6);
	\draw[xshift=5mm,yshift=5mm]
		(0,1) node{$e_1-e_2$}
		(0,2) node{$e_1-e_3$}
		(0,3) node{$e_1-e_4$}
		(0,4) node{$e_1-e_5$}
		(0,5) node{$e_1-e_6$}
		(1,2) node{$e_2-e_3$}
		(1,3) node{$e_2-e_4$}
		(1,4) node{$e_2-e_5$}
		(1,5) node{$e_2-e_6$}
		(2,3) node{$e_3-e_4$}
		(2,4) node{$e_3-e_5$}
		(2,5) node{$e_3-e_6$}
		(3,4) node{$e_4-e_5$}
		(3,5) node{$e_4-e_6$}
		(4,5) node{$e_5-e_6$};
		
\end{scope}
\end{tikzpicture}
\caption{The Dyck path $D(J)$ corresponding to the order filter $J$ in the root poset of type $A_5$ whose minimal elements are $e_1-e_2$ and $e_3-e_5$.}
\end{center}
\end{figure}
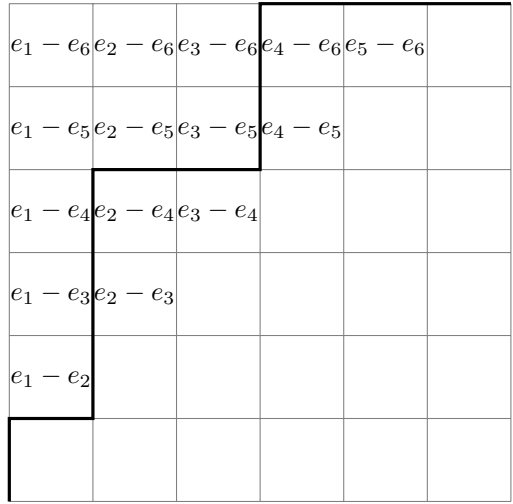
\\
Note that a geometric chain $\J$ of $1$ order filter in the root poset is just a single order filter $J$.
The Dyck path $D(J)$ corresponding to the order filter $J$ is the Dyck path which satisfies
\[\text{The box $[i,j]$ is above $D(J)$ if and only if $e_i-e_j\in J$}.\]
Now $\ind(J)$ is exactly the set of minimal elements of $J$ \cite{athanasiadis05refinement}. Thus we have $e_i-e_j\in\ind(J)$ if and only if $(i,j)$ is a valley of $D(J)$.\\
\\
Take $w\in W$. We have
\begin{align*}
 &w(\ind(J))\subseteq\Phi^+\\
 &\Leftrightarrow w(e_i-e_j)\in\Phi^+\text{ whenever $(i,j)$ is a valley of }D(J)\\
 &\Leftrightarrow w(i)<w(j)\text{ whenever $(i,j)$ is a valley of }D(J)\\
 &\Leftrightarrow (w,D(J))\text{ is a diagonally labelled Dyck path}.
\end{align*}
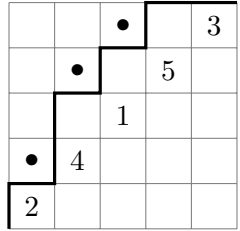
\begin{figure}[h]
\begin{center}
\begin{tikzpicture}[scale=.6]
\begin{scope}
	\draw[gray] (0,0) grid (5,5);
	\draw[very thick] (0,0)--(0,1)--(1,1)--(1,3)--(2,3)--(2,4)--(3,4)--(3,5)--(5,5);
	\draw[xshift=5mm,yshift=5mm]
		(0,0) node{\large{$2$}}
		(1,1) node{\large{$4$}}
		(2,2) node{\large{$1$}}
		(3,3) node{\large{$5$}}
		(4,4) node{\large{$3$}}
		(0,1) node{\large{$\bullet$}}
		(1,3) node{\large{$\bullet$}}
		(2,4) node{\large{$\bullet$}};
\end{scope}
\end{tikzpicture}
\caption{The diagonally labelled Dyck path $(w,D(J))$ where $w=24153$ and $\ind(J)=\{e_1-e_2,e_2-e_4,e_3-e_5\}$. The valleys of $D(J)$ are marked by dots.}
\end{center}
\end{figure}
\begin{lemma}
 The map 
 \begin{align*}
  \ptd:\park&\rightarrow\D_n\\
  [w,J]&\mapsto (w',D(J)),
 \end{align*}
 where $w'\in wW_J$ is the unique representative with $w'(\ind(J))\subseteq\Phi^+$, is a bijection.
 \begin{proof}
  The map $J\mapsto D(J)$ is a bijection from order filters in the root poset of type $A_{n-1}$ to Dyck paths of length $n$. 
  The map $\ptd$ is well-defined since $[w_1,J_1]=[w_2,J_2]$ implies $J_1=J_2$ and $w_1'=w_2'$, so $\ptd([w_1,J_1])=\ptd([w_2,J_2])$.
  We see that $(w',D(J))\in\D_n$ since $w'(\ind(J))\subseteq\Phi^+$. We see that $\ptd$ is injective since $\ptd([w_1,J_1])=\ptd([w_2,J_2])$ implies $D(J_1)=D(J_2)$, so that $J_1=J_2$.
  Furthermore $w_1'=w_2'$, so that $w_1W_{J_1}=w_2W_{J_2}$ and thus $[w_1,J_1]=[w_2,J_2]$.
  We see that $\ptd$ is surjective since for $(w,D)\in\D_n$ we have $(w,D)=\ptd([w,J])$ where $D=D(J)$.
 \end{proof}
\end{lemma}
\noindent
Since $W_J$ is generated by the transpositions $(ij)$ such that $(i,j)$ is a valley of $D(J)$ and the condition $w'(\ind(J))\subseteq\Phi^+$ is equivalent to $w'(i)<w'(j)$ whenever $(i,j)$ is a valley of $D(J)$ we can get $w'$ from $w$ with a simple sorting procedure:
for all maximal chains of indices $i_1<i_2<\ldots<i_l$ such that $(i_j,i_{j+1})$ is a valley of $D(J)$ for all $j\in[l-1]$ sort the values of $w$ on positions $i_1,i_2,\ldots,i_l$ increasingly.
The result is $w'$. From this we also get the $S_n$-action on $\D_n$ that turns $\ptd$ into an $S_n$-isomorphism: for $u\in S_n$ define 
$$u\cdot(w,D):=((uw)',D)$$
where $(uw)'$ arises from $uw$ through the sorting procedure desribed above. Note the analogy between this action and the $S_n$-action on $\pf_n$ in terms of vertically labelled Dyck paths that was described in Section \ref{vert}.\\
\\
One may also view diagonally labelled Dyck paths as a combinatorial model for Shi alcoves, or equivalently $(n+1)$-stable affine permutations. The following lemma makes this explicit.
\begin{lemma}\label{atd}
 The map
 \begin{align*}
  \atd:\widetilde{S}_n^{n+1}&\rightarrow\D_n\\
  \waf&\mapsto(w,D),
 \end{align*}
 where $\waf\ac\in wC$ and $D=D(J)=D(\phi(\rd))$ is the Dyck path corresponding to the order filter corresponding to the dominant Shi region $\rd$ containing $w^{-1}\waf\ac$, is a bijection. Furthermore $\atd=\ptd\circ\pta^{-1}$.
 \begin{proof}
  An immediate check from the definitions of $\ptd$ and $\pta$.
 \end{proof}

\end{lemma}
\noindent
\subsection{The zeta maps are equivalent}
The following theorem relates our zeta map $\zeta$ from Theorem \ref{zeta} to the zeta map $\zeta_{HL}$ of Haglund and Loehr. Recall that $\ttpf=\pi_{\pf}^{-1}\circ\pi_{\Q}$ is the natural $S_n$-isomorphism from $\Q/(n+1)\Q$ to $\pf_n$.
\begin{theorem}
 If $\Phi$ is of type $A_{n-1}$ and $m=1$, then 
 \begin{align*}
  \zeta_{HL}&=\ptd\circ\zeta\circ\ttpf^{-1}\\
  &=\atd\circ\A_{GMV}^{-1}.
 \end{align*}
 \begin{proof}
  Define $\zeta':=\ptd\circ\zeta\circ\ttpf^{-1}$.
  We also have 
  $$\zeta'=\ptd\circ\zeta\circ\ttpf^{-1}=\ptd\circ\pta^{-1}\circ\A^{-1}\circ\ttpf^{-1}=\atd\circ\A_{GMV}^{-1}$$
  using the definition of $\zeta$, Lemma \ref{atd} and Theorem \ref{GMV}. We will show that $\zeta'$ satisfies the following properties:
  If $\zeta'(P,\sigma)=(w,D)$ then firstly $w=\drw(P,\sigma)$ and secondly $(w,D)$ has a valley labelled $(b,c)$ if and only if $(P,\sigma)$ has a rise labelled $(b,c)$ for all $b$ and $c$.
  As noted as the end of Section \ref{zetahl}, these properties define $\zeta_{HL}$ uniquely, so we deduce that $\zeta'=\zeta_{HL}$.\\
  \\
  First we need to check that if $\zeta'(P,\sigma)=(\atd\circ\A_{GMV}^{-1})(P,\sigma)=(w,D)$, then $w=\drw(P,\sigma)$.
  Equivalently we need to verify that if $\waf\in\wa^{n+1}$ with $\waf\ac\subseteq wC$ and $\A_{GMV}(\waf)=(P_{w_R},\sigma)$ then $w=\drw(P_{w_R},\sigma)$.\\
\begin{figure}[h]
\begin{center}
\begin{tikzpicture}[scale=.7]
\begin{scope}
	\draw[gray] (0,0) grid (7,6);
	\draw[very thick] (0,0)--(0,2)--(1,2)--(1,3)--(2,3)--(2,5)--(5,5)--(5,6)--(7,6);
	\draw (0,0)--(7,6);
	\draw[xshift=5mm,yshift=5mm]
		(0,0) node{\large{-6}}
		(1,0) node{\large{-12}}
		(2,0) node{\large{-18}}
		(3,0) node{\large{-24}}
		(4,0) node{\large{-30}}
		(5,0) node{\large{-36}}
		(6,0) node{\large{-42}}
		(0,1) node[fill=gray!40,circle]{\large{1}}
		(1,1) node{\large{-5}}
		(2,1) node{\large{-11}}
		(3,1) node{\large{-17}}
		(4,1) node{\large{-23}}
		(5,1) node{\large{-29}}
		(6,1) node{\large{-35}}
		(0,2) node{\large{8}}
		(1,2) node[fill=gray!40,circle]{\large{2}}
		(2,2) node{\large{-4}}
		(3,2) node{\large{-10}}
		(4,2) node{\large{-16}}
		(5,2) node{\large{-22}}
		(6,2) node{\large{-28}}
		(0,3) node{\large{15}}
		(1,3) node{\large{9}}
		(2,3) node[fill=gray!40,circle]{\large{3}}
		(3,3) node{\large{-3}}
		(4,3) node{\large{-9}}
		(5,3) node{\large{-15}}
		(6,3) node{\large{-21}}
		(0,4) node{\large{22}}
		(1,4) node{\large{16}}
		(2,4) node[fill=gray!40,circle,scale=0.78]{\large{10}}
		(3,4) node[fill=gray!40,circle]{\large{4}}
		(4,4) node{\large{-2}}
		(5,4) node{\large{-8}}
		(6,4) node{\large{-14}}
		(0,5) node{\large{29}}
		(1,5) node{\large{23}}
		(2,5) node{\large{17}}
		(3,5) node{\large{11}}
		(4,5) node{\large{5}}
		(5,5) node{\large{-1}}
		(6,5) node{\large{-7}};
	\draw[xshift=5mm,yshift=5mm]
		(-1,0) node[color=blue,]{\large{1}}
		(-1,1) node[color=blue,]{\large{3}}
		(-1,2) node[color=blue,]{\large{4}}
		(-1,3) node[color=blue,]{\large{5}}
		(-1,4) node[color=blue,]{\large{6}}
		(-1,5) node[color=blue,]{\large{2}};
\end{scope}
\end{tikzpicture}
\caption{The vertically labelled $7/6$-Dyck path $\A_{GMV}(\waf)$ for the dominant $7$-stable affine permutation $\waf=[2,7,3,4,5,0]$. We have that $\waf^{-1}=[-4,1,3,4,5,12]$ is affine Grassmanian.
The positive beads of the normalized abacus $\mathsf{A}(\widetilde{\Delta}_{\waf})$ are shaded in gray.}
\end{center}
\end{figure}
  \\
  First suppose that $w=e$ is the identity. That is, $\waf\in\wa^{n+1}_{\mathrm{dom}}$. So $\waf^{-1}$ is \defn{affine Grassmanian}, that is $\waf^{-1}(1)<\waf^{-1}(2)<\ldots<\waf^{-1}(n)$.
  The set of smallest gaps on the runners of the balanced abacus $\mathsf{A}(\Delta_{\waf})$ is 
  $$\{\waf^{-1}(1),\waf^{-1}(2),\ldots,\waf^{-1}(n)\}.$$
  Thus the set of smallest gaps of the normalized abacus $\mathsf{A}(\widetilde{\Delta}_{\waf})$ is 
  $$\{\waf^{-1}(1)-M_{\waf},\waf^{-1}(2)-M_{\waf},\ldots,\waf^{-1}(n)-M_{\waf}\},$$
  where $M_{\waf}$ is the minimal element of $\Delta_{\waf}$.
  This equals the set $\{l_1,l_2,\ldots,l_n\}$ of labels of the boxes to the left of the North steps of the Dyck path $P_{\waf}$.\\
  \\
  Let $(a_1,a_2,\ldots,a_n)$ be the area vector of $P_{\waf}$. Then we have $l_i=na_i+i-1$. Thus $l_i<l_j$ if and only if either $a_i<a_j$ or $a_i=a_j$ and $i<j$. Furthermore, the label of the $i$-th North step of $P_{\waf}$ is $\sigma(i)=\waf(l_i+M_{\waf})$.
  So the $j$-th label being read in the diagonal reading word is $\drw(P_{\waf},\sigma)(j)=\waf(l_i+M_{\waf})$, where $l_i$ is the $j$-th smallest element of $\{l_1,l_2,\ldots,l_n\}$.
  But the $j$-th smallest element of 
  $$\{l_1,l_2,\ldots,l_n\}=\{\waf^{-1}(1)-M_{\waf},\waf^{-1}(2)-M_{\waf},\ldots,\waf^{-1}(n)-M_{\waf}\}$$
  is just $\waf^{-1}(j)-M_{\waf}$, so $\drw(P_{\waf},\sigma)(j)=\waf(\waf^{-1}(j)-M_{\waf}+M_{\waf})=j$.
  Thus $\drw(P_{\waf},\sigma)=e$, as required.\\
  \\
  In general if $\waf\ac\subseteq wC$ then $\waf=w\waf_D$, where $\waf_D\in\wa^{n+1}_{\mathrm{dom}}$ using Lemma \ref{ratdom}.
  We have $\Delta_{\waf_D}=\Delta_{\waf}$ and thus also $M_{\waf_D}=M_{\waf}$ and $\widetilde{\Delta}_{\waf_D}=\widetilde{\Delta}_{\waf}$.
  Therefore $P_{\waf}=P_{\waf_D}$ and the tuple $(l_1,l_2,\ldots,l_n)$ is also the same for $\waf$ and $\waf_D$.
  Thus the $j$-th label being read in the diagonal reading word of $\A_{GMV}(\waf)=(P_{\waf},\sigma)$ is 
  $$\drw(P_{\waf},\sigma)(j)=\waf(\waf_D^{-1}(j)-M_{\waf}+M_{\waf})=w(j).$$
  So $\drw(P_{\waf},\sigma)=w$, as required.\\
  \\
  The second property we need to check is that if $\zeta'(P,\sigma)=(\ptd\circ\zeta\circ\ttpf^{-1})(P,\sigma)=(w,D)$ then $(w,D)$ has a valley labelled $(a,b)$ if and only if $(P,\sigma)$ has a rise labelled $(a,b)$.
  But this follows from general considerations: since $\zeta'=\ptd\circ\zeta\circ\ttpf^{-1}$ is a composition of $S_n$-isomorphisms it is itself an $S_n$-isomorphism.
  In particular, the $S_n$-stabilizers of $(P,\sigma)$ and $(w,D)$ must agree. But $(P,\sigma)$ has a rise labelled $(a,b)$ if and only if $b$ is the smallest integer with $a<b\leq n$ such that the transposition $(ab)$ fixes $(P,\sigma)$,
  and similarly $(w,D)$ has a valley labelled $(a,b)$ if and only if $b$ is the smallest integer with $a<b\leq n$ such that the transposition $(ab)$ fixes $(w,D)$.
  Thus $(w,D)$ has a valley labelled $(a,b)$ if and only if $(P,\sigma)$ has a rise labelled $(a,b)$.
  This concludes the proof.
 \end{proof}

\end{theorem}
\section{Outlook}
Given the algebraically defined uniform zeta map $\zeta$ that has the combinatorial interpretation $\zeta_{HL}$ if $\Phi$ is of type $A_{n-1}$ and $m=1$ it is natural to ask for combinatorial interpretations in other types and for larger $m$.
The extended abstract \cite{sulzgruber14typec} supplies an answer for type $C_n$ and $m=1$. For types $B_n$ and $D_n$ similar combinatorial zeta maps will be defined in future work \cite{sulzgruber16parking}.
\section{Acknowledgements}
The author would like to thank Robin Sulzgruber for many interesting conversations and Nathan Williams for encouraging him to write up these results.
Furthermore, the author would like to thank the anonymous referee for their careful reading and helpful comments.
The research was supported by the Austrian Science Foundation FWF, grant S50-N15 in the framework of the Special
Research Program ``Algorithmic and Enumerative Combinatorics'' (SFB F50).

\bibliographystyle{alpha}
\bibliography{literature}

\end{document}